\DeclareFontFamily{U}{wncy}{}
\DeclareFontShape{U}{wncy}{m}{n}{%
   <5>wncyr5%
   <6>wncyr6%
   <7>wnyr7%
   <8>wncyr8%
   <9>wncyr9%
   <10>wncyr10%
   <11>wncyr10%
   <12>wncyr6%
   <14>wncyr7%
   <17>wncyr8%
   <20>wncyr10%
   <25>wncyr10}{}

\documentclass[a4paper,11pt]{article}
\usepackage{graphicx}
\usepackage{latexsym,color}
\usepackage{amsmath,amsfonts,amssymb,amsthm}
\newtheorem{thm}{Theorem}[section]
\newtheorem{lem}[thm]{Lemma}

\newtheorem{cor}[thm]{Corollary}
\newtheorem{prop}[thm]{Proposition}

\newtheorem{rem}[thm]{Remark}

\input epsf

\title{Weyl-Mahonian Statistics for Weighted Flags of Type A-D}
\author{Roland Bacher\footnote{This work has been partially supported by the LabEx PERSYVAL-Lab (ANR--11-LABX-0025). The author is a member of the project-team GALOIS supported by this LabEx.}}

\begin{document}
\maketitle

%\par flags1.tex dans recherche/reseaux (mahonian1.tex est une copied'une version ancienne contenant du materiel sur les polynomes).

%cf computations in ifnode2/bacher/brauer....mw

{\sl Abstract\footnote{
Math. class:  Primary: 05A05, Secondary: 05A10, 05A15, 05E15, 20G15.}: 
We relate properties of 
weighted flags (or multiflags) of type A-D
to statistics of the corresponding Weyl groups.
For type A, we recover the Mahonian statistics on
symmetric groups. Finally, we sketch briefly an easy extension
incorporating statistics for so-called Euler-polynomials.
}
\vskip.5cm

\section{Introduction}

Combinatorial statistics (or statistics for short) are combinatorially defined
$\mathbb N$-valued functions on sets of combinatorial elements. This paper
deals with three interesting statistics on Weyl groups of the three 
infinite families A, B-C (giving rise to isomorphic Weyl groups) and D. 
Weyl groups of type A are finite symmetric groups.
A well-known and natural statistic on such groups is defined by 
the number of inversions of a permutation $\sigma$ acting on
the totally ordered set $\{1,\dots,d\}$. It gives the length of $\sigma$
with respect to the usual 
Coxeter generators consisting of the $d-1$ transpositions $(1,2),(2,3),\dots,(d-1,d)$ exchanging two consecutive integers. 
Another important statistic is given by summing indices of descents 
of a permutation. This statistic is called the Major statistic in honour
of Major Mac Mahon who proved equidistribution of the length and the Major
statistic for symmetric groups, see \cite{MM13}.
A third interesting statistic is given by Eulerian polynomials encoding
numbers of permutations with a given number of descents.

All these three statistics arise naturally when counting
weighted flags over finite fields. This observation
allows an extension to all Weyl groups of the infinite
families A,B-C,D (with a caveat for type D:
numbers of descents have to be modified slightly).
Techniques coming from the theory of linear groups, 
outlined only briefly, should allow
to treat the cases of the exceptional Weyl groups.

The main part of this paper deals with the construction of 
the joint statistic for inversion numbers and Major indices.
These statistics are encoded by so-called Weyl-Mahonian polynomials.
Descent numbers are a cheap bonus outlined in a last chapter.

There exists several generalizations of 
the above statistics, mainly to Weyl groups of type BC,
see for example the incomplete liste \cite{ABR01}, \cite{ABR06}, \cite{AR}, \cite{Ba}, \cite{BeBr}, \cite{Bia03}, 
\cite{BiaZe}, \cite{Br}, \cite{Ca54}, \cite{Ca75}, \cite{C2003}, \cite{C2008},
\cite{CG}, 
\cite{CF94}, \cite{CF95}, \cite{CF95b},
\cite{CM}, \cite{DW}, \cite{Eu}, \cite{F10}, \cite{FS},
\cite{Hy},  
%\cite{LP}, 
\cite{MR}, \cite{Re93}, \cite{Rei93}, \cite{Rei93a}, \cite{Rei95}, \cite{St76}, \cite{SW}, \cite{St} of related works. 
Our paper is an addition to this list.

Weyl groups are a special case of Coxeter groups, studied for example 
in the monographs \cite{BB01}, \cite{Da}, \cite{DH}, \cite{H90}.
Standard books of enumerative combinatorics are \cite{MM}, \cite{St86}, \cite{St99}. Finally, \cite{Bri} and \cite{H72} are good introductions to 
flag-varieties.

\section{Main results for flags of type A}\label{sectmainresults}

Vector-spaces are always finite-dimensional in the sequel.

A \emph{(partial) flag} of a vector-space $V$ over a field $\mathbb F$
is a sequence of subspaces 
$$\{0\}=V_0\subset V_1\subset\dots\subset V_{k-1}\subset V_k\subset V$$
of strictly increasing dimensions $0=\dim(V_0)<\dim(V_1)<\dots <\dim(V_{k-1})
<\dim(V_k)\leq \dim(V)$. We omit henceforth
the trivial subspace $V_0=\{0\}$ and use the notation 
$V_1\subset\dots\subset V_k$ for a flag of $V$.

A \emph{weighted flag} is a flag $V_1\subset \dots\subset V_k$
together with a sequence $w_1,\dots,w_k$
%\in\{1,2,\dots\}^k$ 
of strictly positive integers attached to the subspaces
$V_1,\dots,V_k$. We call the sequence $w_1,\dots,w_k$
the \emph{weight-sequence} of the weighted flag
$F=(V_1\subset\dots \subset V_k;w_1,\dots,w_k)$.
The \emph{weight} of such a weighted flag $F$ is 
defined as $w(F)=\sum_{i=1}^kw_i\dim(V_i)$. It is the content of the
partition, called the \emph{weight-partition}, 
with $w_i$ parts of size $\dim(V_i)\geq i$ for $i=1,\dots,k$.

\begin{rem}\label{remmultiflag}
The following definition, equivalent to weighted flags, 
avoids weights and considers instead
finite sequences of weakly increasing subspaces: A
partial \emph{multiflag} is a weakly
increasing finite sequence $\{0\}\not=V_1\subset V_2\subset\dots
\subset V_k$ of a vector space (finite-dimensional, as always).
Partial multiflags are in bijection with weighted flags:
Repeat each part $V_i$ of a weighted flag $w_i$ times.
The weighted flag $F=(V_1\subset \dots\subset V_k;w_1,\dots,w_k)$
corresponds thus to the weighted multiflag 
\begin{align*}&W_1=\dots=W_{w_1}=V_1\subset W_{w_1+1}=\dots=W_{w_1+w_2}=V_2
\subset\\
&\subset  W_{w_1+w_2+1}=\dots=W_{w_1+w_2+w_3}=V_3\subset \dots\subset W_l=V_k\end{align*}
consisting of $l=\sum_{i=1}^k w_i$ weakly increasing subspaces (with 
allowed equal consecutive subspaces
$W_i=W_{i+1}$). We have of course
$w(F)=\sum_{i=1}^kw_i\dim(V_i)=\sum_{i=1}^l\dim(W_i)$.
\end{rem}

We denote by $\mathcal{F}(V)$ the set of all flags and
by $\mathcal{WF}(V)$ the set of all weighted flags of
a vector-space $V$. 

An \emph{inversion} of a  permutation $\sigma\in\mathcal S_d$ 
acting on $\{1,\dots,d\}$ is given by $1\leq i<i\leq d$ 
such that $\sigma(i)>\sigma(j)$.
We write $\mathrm{inv}(\sigma)=\vert \{i,j\vert 
i<j,\sigma(i)>\sigma(j)\}\vert$ for the number of inversions
of a permutation $\sigma$ in $\mathcal S_d$.
It is well-known that $\mathrm{inv}(\sigma)$ corresponds to the 
length $l(\sigma)$
of $\sigma$ in terms of the Coxeter generators $(1,2),\dots,(d-1,d)$,
see for example Proposition \ref{propsignedlength} which generalizes
this result to hyperoctahedral groups (signed permutation groups).

A \emph{descent} of a permutation $\sigma$ in $\mathcal S_d$ is a value 
$i<d$ such that $\sigma(i)>\sigma(i+1)$.
The \emph{Major index} 
\begin{align}\label{defmajind}
\mathrm{maj}(\sigma)&=\sum_{i,\sigma(i)>\sigma(i+1)}i
\end{align}
sums up the indices of all descents of a permutation $\sigma$ in $\mathcal S_d$.

The following result relates Mahonian statistics over $\mathcal S_d$ 
with statistics of weighted flags over finite fields:

\begin{thm}\label{mainthm} We have
\begin{align}\label{formofmainthm}
\sum_{F\in \mathcal{WF}(\mathbb F_q^d)}t^{w(F)}&=M_d
\prod_{j=1}^d\frac{1}{1-t^j}
\end{align}
with 
$$M_d=\sum_{\sigma\in \mathcal S_d}
q^{\mathrm{inv}(\sigma)}t^{\mathrm{maj}(\sigma)}$$
denoting the Mahonian statistics of $\mathcal S_d$ encoding
multiplicities of elements of given length and Major index.
\end{thm}

Observe that the sum over elements in $\mathcal{WF}(\mathbb F_q^d)$ is the
sum over weighted flags in an abstract $d$-dimensional 
vector-space over $\mathbb F_q$. It does not depend on the choice 
of a basis for $\mathbb F_q^d$. Observe also that Theorem \ref{mainthm}
involves a notational abuse: the integer $q$ representing the number of
elements of a finite field
$\mathbb F_q$ (which is unique, up to isomorphism) 
should be considered as a formal variable. Equalities are 
among formal power series in $q$ and $t$. Easy majorations
show however that we get converging series in small neighbourhoods
of $(0,0)\in \mathbb C^2$.

The factor $\prod_{j=1}^d(1-t^j)^{-1}$ is the 
generating series for partitions involving at most $d$ parts
(or, equivalently, for partitions involving only parts of length at most
$d$).

An obvious modification of Theorem \ref{mainthm} and its generalizations 
holds for arbitrary fields by considering powers of $q$ as dimensions
of a cellular decomposition (corresponding to the Bruhat decomposition of 
simple algebraic groups of Lie type) 
of the set of all weighted flags
into cells carrying affine structures over a field $\mathbb F$.
We stick to the elementary enumerative approach over finite fields
for simplicity.

We denote by ${d\choose k}_q$ the $q$-binomial coefficient defined 
(for example) recursively by
${d\choose 0}_q={d\choose d}_q=1$
and ${d\choose i}_q={d-1\choose i-1}_q+q^i{d-1\choose i}_q$.

Exploiting the geometric structure on the left hand side of 
(\ref{formofmainthm}) we get:

\begin{cor}\label{correc} The polynomials
$M_d=\sum_{\sigma\in \mathcal S_d}
q^{\mathrm{inv}(\sigma)}t^{\mathrm{maj}(\sigma)}$ are recursively defined by 
$M_0=1$ and by 
\begin{align}\label{formularecMd}
M_d&=\sum_{i=0}^{d-1}t^i\left(\prod_{j=i+1}^{d-1}1-t^j\right)
{d\choose i}_qM_i
\end{align}
(using the convention $\prod_{j=d}^{d-1}(1-t^j)=1$).
\end{cor}

Evaluating (\ref{formularecMd}) at $t=1$ yields the
recursion $L_d={d\choose d-1}_qL_{d-1}=(1+q+q^2+\cdots+q^{d-1})L_{d-1}$ 
and implies the well-known factorization
\begin{align}\label{formulalengthfactorSd}
L_d&=
\prod_{j=1}^d\frac{1-q^j}{1-q}
\end{align}
for the generating polynomial $L_d=\sum_{\sigma\in\mathcal S_d}q^{l(\sigma)}$ 
enumerating elements of $\mathcal S_d$
accordingly to their length with respect to the Coxeter generators
$(i,i+1)$.

Evaluating $M_d$, given by Formula (\ref{formularecMd}) of Corollary
\ref{correc} at $q=1$ shows equidistribution of the length statistic with
the Major statistic:

\begin{prop}\label{propevqeqoneofMd}
The evaluation at $q=1$ of $M_d$ factorizes as
$$\prod_{j=1}^d\frac{1-t^j}{1-t}\ .$$
\end{prop}

The well-known symmetry $M_d(q,t)=M_d(t,q)$ in $q,t$ of the polynomials $M_d$ 
is however not obvious from Formula (\ref{formularecMd}) 
in Corollary \ref{correc}.

The sequel of the paper is organized as follows:

In Sections \ref{sectflagsBC} and \ref{sectflagD} we describe analogues of
Theorem \ref{mainthm}
and Corollary \ref{correc} for Weyl groups of type B,C and D
(symmetric groups are of course Weyl groups of type A).

Section \ref{sectalgapproachforexctypes} sketches briefly an approach
for dealing with exceptional Weyl groups.

Sections \ref{sectdefflags}-\ref{sectproofD}
are devoted to complements and proofs.

Section \ref{sectalgapproachforexctypes}
contains a brief and sketchy description perhaps useful when dealing
with exceptional Weil groups (of type E,F,G).

Finally, Section \ref{sectEulerpol}
describes generalizations taking
also into account the number $\sum_{i=1}^k w_i$ of a weighted flag
$(V_1\subset\dots\subset V_k;w_1,\dots,w_k)$, the dimension $\dim(V_k)$
of the last subspace in a flag and statistics related to 
signs in the case of Weyl groups of type BC and D. Statistics of these 
numbers in the case of type A or BC involve  Euler polynomials counting
descents. Details are easy to fill in and are mostly omitted.

%\begin{rem} It is also possible to consider sums
%$$\sum_{(V_1,\dots,V_k)\in\mathcal F(\mathbb F_q^d)}t^{\dim(V_1)+\dots
%+\dim(V_k)}\ .$$
%These sums (and their analogues for other types) will be considered in a 
%separate paper.
%\end{rem}

\section{Main results for flags of type B and C}\label{sectflagsBC}

We consider a finite-dimensional vector space $V$ over a field
of characteristic $\not= 2$ (this avoids technical problems in the 
symmetric case)
endowed with a non-degenerate symmetric or antisymmetric bilinear form $b$.
A subspace $L$ of $V$ is \emph{isotropic} if the restriction of $b$ to 
$L\times L$ is zero. We suppose that $V$ has isotropic subspaces
of the maximal dimension $\lfloor \dim(V)/2\rfloor$ 
compatible with non-degeneracy of $b$. Such a space is called a \emph{symplectic
space} if $b$ is antisymmetric. Symplectic spaces will always
be denoted by $(V,\omega)$. \emph{Lagrangians} are maximal 
isotropic subspaces of symplectic spaces. 

A space $(V,b)$ is of type C
if $V$ is of even dimension and $b$ is a non-degenerate symplectic form, 
of type B if $(V,b)$ is a quadratic space of odd dimension
and of type D if $(V,b)$ is a quadratic space of even dimension.
Type B and C share Weyl groups: The common Weyl group of 
$(\mathbb F^{2d},\omega)$ and of  $(\mathbb F^{2d+1},b)$ (with $b$ a suitable symmetric bilinear form) is given by the hyperoctahedral 
group $\mathcal S_d^\pm$ of all permutations $\sigma$
of $\{\pm 1,\dots,\pm d\}$ such that $\sigma(-i)=-\sigma(i)$ for $i=1,\dots,d$.
The case of $(\mathbb F^{2d},b)$ (with $b$ a suitable symmetric bilinear form
on a vector-space of even dimension $2d$)
corresponds to the subgroup $\mathcal S_d^D$ of index two in $\mathcal S_d^\pm$
consisting of all elements $\sigma$ in $\mathcal S_d^\pm $ such that 
$\prod_{i=1}^d\sigma(i)=d!$.

A flag of $(V,b)$ is a flag $V_1\subset \dots\subset V_k$ consisting 
only of isotropic subspaces. Weighted flags of $(V,b)$ and weights of flags
are defined in the obvious way. It is of course again possible to 
replace weighted flags by multiflags involving only isotropic subspaces
(ending with an even isotropic subspace in the case of type D),
see Remark \ref{remmultiflag}. 
We denote by $\mathcal{WF}(V,b)$ the set of all 
weighted flags of $(V,b)$.

In order to state analogues of Theorem \ref{mainthm}
and Corollary \ref{correc} we need to introduce the relevant 
statistics on the corresponding Weyl groups. We call 
these statistics \emph{Weyl-Mahonian} since they extend Mahonian
statistic on symmetric groups to other Weyl groups.

We start by introducing a somewhat exotic order-relation,
denoted by $<_\pm $, on the set $\mathbb R$ of real numbers.
It is defined by
$x<_\pm x+\epsilon<_\pm 0<_\pm -x-\epsilon<_\pm -x$ for strictly positive
$x$ and $\epsilon$. 
The induced order relation on $\mathbb Z$ is given 
by
\begin{align}\label{pmorderonZ}
1<_\pm 2<_\pm 3<\pm \dots<_\pm 0<_\pm \dots <_\pm -3<_\pm -2<_\pm -1
\end{align}
or equivalently by
$$(\mathbb N^*,<)<_\pm 0<_\pm (-\mathbb N^*,<)$$
with $(\mathbb N^*,<)$ (respectively $(-\mathbb N^*,<)$ denoting
the ordered set $\mathbb N^*=\mathbb N\setminus\{0\}$ (respectively
$-\mathbb N^*\setminus\{0\}$) endowed with the usual order. 
In particular, $(\mathbb Z,<_\pm)$ is totally ordered with smallest element $1$
and largest element $-1$. Observe however that $(\mathbb Z,<_\pm)$ is 
not well-ordered: $\{-1,-2,-3,\dots\}$ has no smallest element.

We recall that $\mathcal S^{\pm}_d$ denotes the hyperoctahedral group 
of all $2^d\cdot d!$ signed permutations of 
$\{\pm 1,\dots,\pm d\}$. We consider it as the Weyl group of type C 
generated by the transpositions $(1,2),(2,3),\dots,(d-1,d)$ 
(where $(i,i+1)$ denotes the signed permutation $j\longmapsto j$ if $j\not\in\{\pm i,\pm(i+1)\}$ and exchanging the pair $\pm i$ with the pair $\pm (i+1)$
by preserving signs) and by the sign change $(d,-d)$ transposing 
the elements of the largest pair $\{d,-d\}$ of opposite integers.

The length-function of an element in $\mathcal S_d^\pm$
with respect to these $d$ generators is given by the following 
(surely well-known) result:

\begin{prop}\label{propsignedlength} The length $l^\pm(\sigma)$ of an element $\sigma\in 
\mathcal S_d^\pm$ with respect to the generators
$(1,2),(2,3),\dots,(d-1,d),(d,-d)$ is given by the formula
\begin{align}\label{defsignedlength}
l^\pm(\sigma)&=\sum_{0<i<j,\sigma(i)>_\pm \sigma(j)}1+
\sum_{0<i,\sigma(i)<0}(d+1+\sigma(i))\ .
\end{align}
\end{prop}

Proposition 3.1 of \cite{Br} gives a different formula 
(with respect to a slightly different 
generating set).

We call a pair $0<i<j\leq d$ with 
$\sigma(i)>_\pm \sigma(j)$
an \emph{inversion} of $\sigma$.
In order to avoid the use of the somewhat exotic order relation 
$<_\pm$ defined by (\ref{pmorderonZ})
one can replace the condition $\sigma(i)>_\pm \sigma(j)$ with the 
equivalent condition $\sigma(i)\sigma(j)(\sigma(i)-\sigma(j))>0$.

Observe that Formula (\ref{defsignedlength}) boils down to
Formula (\ref{defmajind}) for elements in the subgroup $\mathcal S_d
\subset \mathcal S_d^\pm$ of ordinary
(unsigned) permutations.

The Weyl-Major index of an element $\sigma$ in  $\mathcal S^\pm_d$ 
is defined by 
\begin{align}\label{defsignedmaj}
\mathop{Wmaj}(\sigma)&=\sum_{0<i,\sigma(i)>\sigma(i+1)}i+\sum_{i>0,\sigma(i)<0}1\ .
\end{align}

We call the polynomial
\begin{align}\label{defsignedMd}
M_d^\pm=M_d^\pm(q,t)&=\sum_{\sigma\in\mathcal S_d^{\pm}}q^{l^\pm(\sigma)}t^{\mathop{Wmaj}(\sigma)}
\end{align}
the \emph{Weyl-Mahonian statistics} of $\mathcal S_d^\pm$.
It encodes the number of elements of the hyperoctahedral group 
$\mathcal S_d^\pm$ with given length and Weyl-Major index.

We have:

\begin{thm}\label{thmsympl} We have
$$\sum_{F\in\mathcal{WF}(\mathbb F_q^{2d},\omega)}t^{w(F)}=M_d^{\pm}
\prod_{j=1}^d\frac{1}{1-t^j}\ .$$
\end{thm}

Theorem \ref{thmsympl} generalizes Theorem 
\ref{mainthm} in the following sense: Suppose 
that the first $d$ basis elements of $\mathbb F^{2d}$
span a Lagrangian $L$ in $\mathbb F^{2d}$.
Weighted symplectic flags contained in $L$ correspond
to weighted ordinary flags of $\mathbb F^d$ and are 
enumerated by restricting the sum 
in (\ref{defsignedMd}) to the subgroup $\mathcal S_d$
of ordinary permutations.

Weyl groups of type B are also covered by Theorem
\ref{thmsympl} as shown by the following result:

\begin{thm}\label{thmBequalC} We have
$$\sum_{F\in\mathcal{WF}(\mathbb F_q^{2d},\omega)}t^{w(F)}=
\sum_{F\in\mathcal{WF}(\mathbb F_q^{2d+1},Q)}t^{w(F)}$$
where $\mathcal{WF}(\mathbb F_q^{2d+1},Q)$ denotes the set of
weighted flags over a non-degenerate quadratic space 
$(\mathbb F_q^{2d+1},Q)$ of odd dimension $2d+1$ 
over a finite field  $\mathbb F_q$ of odd characteristic 
such that the quadratic form 
$Q$ admits a $d$-dimensional isotropic subspace.
\end{thm}

Theorem \ref{thmsympl} implies the following corollary 
which expresses the Weyl-Mahonian statistics $M_d^{\pm}$ on
$\mathcal S_d^\pm$ in terms of
ordinary Mahonian statistic $M_0,\dots,M_d$ (given recursively by 
Corollary \ref{correc}) on symmetric groups:

\begin{cor}\label{corformulaMdpm} We have
$$M_d^{\pm}=\sum_{k=0}^dt^k\left(\prod_{j=0}^{k-1}\frac{1-q^{2d-2j}}{1-q^{k-j}}\right)
\left(\prod_{j=k+1}^d(1-t^j)\right)M_k\ .$$
\end{cor}

Corollary \ref{corformulaMdpm} is of course an analogon of Corollary
\ref{correc} for hyperoctahedral groups.

\begin{rem} The function defined by (\ref{defsignedmaj}) is called
Weyl-Mahonian index in order to avoid confusion with the flag-Major
index of \cite{AR} extending the Major
index in a different way to $\mathcal S_d^\pm$ (and more generally to 
wreath-products of $\mathcal S_d$ with finite cyclic group). 
\end{rem}

\section{Main results for flags of type D}\label{sectflagD}

We denote by $\mathcal H=\mathcal H(\mathbb F)$ the hyperbolic plane
over a field $\mathbb F$
realized as the quadratic space $\mathbb F^2$ endowed with a norm given 
by the quadratic form $(x,y)\longmapsto xy$.

We denote by $\mathcal I$ a fixed maximal $d$-dimensional 
isotropic subspace (also called a \emph{metabolizer}) 
of $\mathcal H^d$. As always, a 
flag $F=(V_1\subset \dots\subset V_k)$ is again a strictly increasing sequence
of non-trivial isotropic subspaces of $\mathcal H^d$ (with $\mathcal H^d$
denoting the orthogonal sum of $d$ copies of $\mathcal H$). The 
\emph{$\mathcal{I}$-parity} (or simply the \emph{parity}) of a flag 
ending with $V_k$ is the parity of the integer $\dim(V_k/(V_k\cap \mathcal I))
=\dim(V_k)-\dim(V_k\cap \mathcal I)$. Flags of even parity are simply
called even flags. We denote by ${\mathcal F}^e(\mathcal H^d)$
the set of all even flags and by $\mathcal{WF}^e(\mathcal H^d)$
the set of all weighted even flags.

We consider now Weyl groups of type D, given by the subgroup
$\mathcal S_d^D$ of the hyperoctahedral $\mathcal S_d^\pm$ consisting of
all signed permutations $\sigma$ such that $\prod_{i=1}^d\sigma(i)=d!$.
Given an element $\sigma\in\mathcal S_d^D$ we set 
\begin{align}\label{deflengthD}
l^D(\sigma)&=\sum_{0<i<j,\sigma(i)>_\pm \sigma(j)}1+
\sum_{0<i,\sigma(i)<0}(d+\sigma(i))\ .
\end{align}
We will see in Proposition \ref{proplengthD} that $l^D$ defines
the natural length function on Weyl groups of type $D$.

The Weyl-Major index $\mathop{Wmaj}(\sigma)$ of elements in $S_d^D$ 
coincides with the Weyl-Major index in $S_d^\pm$ and is also given by 
Formula (\ref{defsignedmaj}).

The generating series of Weyl-Major statistics on $\mathcal S_d^D$ 
has a nice factorization given by the following result:

\begin{thm}\label{mainthmtypeD} We have 
\begin{align}\label{formflagMajDqeqone}
\sum_{\sigma\in \mathcal S_d^D}t^{\mathop{Wmaj}(\sigma)}
&=\frac{(1-t)^d+(1+t)^d}{2}\prod_{j=1}^d\frac{1-t^j}{1-t}\ .
\end{align}
\end{thm}

We call the polynomial
\begin{align}\label{defflagMahD}
M_d^D=M_d^D(q,t)&=\sum_{\sigma\in\mathcal S_d^D}q^{l^D(\sigma)}t^{\mathop{Wmaj}(\sigma)}
\end{align}
the \emph{Weyl-Mahonian statistics} of $\mathcal S_d^D$.
It encodes the number of elements of
$\mathcal S_d^D$ with given length and Weyl-Major index.

We have:
\begin{thm}\label{thmwgtflagsD} We have
$$\sum_{F\in\mathcal{WF}^e(\mathcal H^d_q)}t^{w(F)}=M_d^D
\prod_{j=1}^d\frac{1}{1-t^j}$$
with $\mathcal H_q^d$ denoting the orthogonal sum of 
$d$ hyperbolic planes over a finite field $\mathbb F_q$ of odd characteristic
and with $\mathcal{WF}^e(\mathcal H^d_q)$ denoting the set of 
all weighted even flags (with respect to a fixed maximal isotropic 
subspace $\mathcal I$ of $\mathcal H_q^d$).
\end{thm}

The analogue of Corollaries \ref{correc} and \ref{corformulaMdpm} for type D
is given by:

\begin{cor}\label{corformMD} The polynomials $M_d^D$ for the 
Weyl-Mahonian statistics on
Weyl groups of type $D$ are given by the formula
\begin{align}\label{formulaflagMahtypeD}
M_d^D=\sum_{k=0}^dt^k\left(\prod_{j=k+1}^d(1-t^j)\right){d\choose k}_q
\left(\sum_{l=0}^{\lfloor k/2\rfloor}{k\choose 2l}_qq^{l(2d+2l-2k-1)}\right)M_k
\end{align}
where ${d\choose k}_q$ are $q$-binomials (see for exemple 
Corollary \ref{correc}) and
where $M_k$, defined for example by Corollary \ref{correc}, gives the 
Mahonian statistic on $\mathcal S_k$ .
\end{cor}

\begin{rem}\label{remtwoclassesD} The parity condition in the type D case
is due to the existence of two orbits (defined by the 
parity condition) of maximal isotropic subspaces.
%conjugacy classes of maximal isotropic 
%subspaces for the corresponding linear groups. 
%One of the conjugacy classes corresponds
%to complete even flags. The remaining class is given by
%complete odd flags.
\end{rem}

\section{Modifications for exceptional types}\label{sectalgapproachforexctypes}

We denote by $G$ a simple group of Lie type over an algebraically closed field (say over $\mathbb C$ for simplicity).
Borel subgroups are maximal connected solvable subgroups of $G$.
We fix a Borel subgroup $B_0$ and a maximal 
torus $T_0$ contained in $B_0$.
The Weyl group $W$ of $G$ is the finite quotient 
group $N_0/T_0$ with $N_0$ denoting
the normalizer of $T_0$ in $G$. We get thus a Bruhat decomposition 
$G=\bigcup_{w\in W}B_0wB_0$ with $W$ denoting the Weyl group
of $G$, represented by suitable elements of $N_0$. 
A cell $B_0wB_0/B_0$ can
be identified with an affine space whose points index all associated Borel 
subgroups (given by conjugates of $B_0$ by elements of $B_0wB_0$). The dimension of such a cell $B_0wB_0/B_0$
is given by the natural length 
$l^W(w)$ of $w$ with respect to the standard generators associated 
to simple roots.
To each Borel
group $B$ in a cell $B_0wB_0$, we associate the parabolic
subgroup $P_B$ of $G$ generated by $B$ and by representants in $N_0$
of all standard generators $g\in W$ such that $l(wg)=l(w)+1$.
In accordance with our previous terminology we call such a parabolic 
group $P_B$ standard. (Warning: This is not the usual meaning of
``standard'' in the theory of Lie groups.) Its structure depends only
on the cell $B_0wB_0$ under consideration. 
The standard parabolic group associated to $B_0$ is $G$. In the opposite
direction, standard parabolics associated to $B_0wB_0$ for $w$ the Coxeter 
element are simply Borel groups of $B_0wB_0$.

Parabolic subgroups play the role of partial flags and
standard parabolic groups play the role of standard flags. 
The standard weight of a standard parabolic subgroup $P_B$ is defined as the 
sum of dimensions over all non-trivial invariant subspaces
of $\rho(P_B)$ with $\rho$ denoting the adjoint representation.
(Warning: Weights of flags and weights of representations are
of course unrelated.)

An arbitrary parabolic subgroup $P'$ contained in a standard parabolic
group $P_B$ is generated by $B$ and by a subset $\mathcal G'_B\subset \mathcal G_B$ with $B$ and $\mathcal G_B\subset N_0$ (representating a subset of 
standard generators of $W$) generating the standard parabolic subgroup 
$P_B$. A standard parabolic subgroup $P_B$ contains thus $2^{\sharp(\mathcal G_B)}$
different parabolic subgroups containing $B$ 
(corresponding to all flags refining the standard flag associated
to the standard parabolic subgroup $P_B$). 
Weights associated to a parabolic subgroup $P'$ are defined by sequences of strictly positive integers
indexed by invariant subspaces of $\rho(P')$ where $\rho$ denotes
the adjoint representation of $G$.
(The associated weight is then given by $\sum_{i}w_i\dim(V_i)$
where the sum is over all invariant subspaces $V_i$ of $\rho(P')$
with associated weight $w_i\in\{1,2,\dots\}$.)

Working out the technical details of this machinery
turns the determination of the Weyl-Mahonian statistics on the six
exceptional Weyl groups of type $E,F,G$ (for the flags of which there exists 
to my knowledge no obvious easy combinatorial description) 
into finite computations.

\section{Complements for weighted flags}\label{sectdefflags}

The \emph{empty flag} is reduced to the trivial (and omitted) subspace
$V_0=\{0\}$. It is the unique flag of weight $0$ as a weighted flag.

A flag $V_1\subset V_2\subset\dots\subset V_k$ of a finite-dimensional 
vector space $V$ is \emph{complete}
if $k=\dim(V)$ (and thus $\dim(V_i)=i$ for $i=0,\dots,\dim(V)$).
A flag is \emph{maximal} if $\dim(V_k)=\dim(V)$ and \emph{nonmaximal} 
otherwise. Complete flags are maximal.

Every flag $F=(V_1\subset \dots\subset V_k)$ can be endowed with the 
\emph{minimal weight} $\sum_{i=1}^k \dim(V_i)$ defined by the weight-sequence 
$w_1=w_2=\dots=w_k=1$. Any other weight sequence yields a 
strictly larger weight for $F$.

A weighted flag $(V_1\subset \dots \subset V_k;w_1,\dots,w_k)$
of $V$ defines a function $\mu:V\longrightarrow \mathbb N$
by setting $\mu(v)=\sum_{i,v\in V_i}w_i$. We call the value 
$\mu(v)$ the \emph{(weighted) multiplicity} of $v$.
We have of course $\mu(v)=0$
if and only if $v$ is in the complement of $V_k$.

A flag $F=(V_1\subset V_2\subset \dots\subset V_k)$ of a $d$-dimension vector
space $V$ can be encoded by a basis $f_1,\dots,f_d$ such 
that $f_1,f_2,\dots,f_{\dim(V_i)}$ span $V_i$ for all $i$ and by the 
strictly increasing sequence $0<\dim(V_1)<\dim(V_2)<\dots<\dim(V_k)$
of dimensions. If $F$ is weighted with weight sequence $w_1,\dots,
w_k$, the decreasing sequence $\mu(f_1)\geq \mu(f_2)\geq\dots\geq \mu(f_k)$
of multiplicities of basis vectors as above
defines the conjugate partition of the weight partition and we have thus
$$\sum_{i=1}^kw_i\dim(V_i)=\sum_{i=1}^d\mu(f_i)\ .$$
We call the partition defined by $\mu(f_1),\dots,\mu(f_{\dim(V_k)})$
the \emph{conjugate (weight) partition}.
A weighted flag is of course also uniquely defined by 
a basis $f_1,\dots,f_d$ as above and by the conjugate partition
$\mu(f_1),\dots,\mu(f_d)$ involving at most $d$ non-zero parts.

\section{$q$-binomials and proofs of Corollary \ref{correc} and Proposition \ref{propevqeqoneofMd}}

The following lore is easy and well-known (see for example Proposition 
1.3.18 of \cite{St86}):

\begin{prop}\label{propGrassm} The number of $k$-dimensional subspaces of $\mathbb F_q^d$
is given by the $q$-binomial
$${d\choose k}_q=\prod_{j=1}^k \frac{q^{d+1-j}-1}{q^j-1}\ .$$
\end{prop}

\begin{proof}[Proof of Corollary \ref{correc}] We start by sorting weighted flags of $\mathbb F_q^d$ according to the dimension
$i=\dim(V_\omega)\in\{0,1,\dots,d-1\}$ of their last and largest subspace $V_\omega$.
By Proposition \ref{propGrassm} there are ${d\choose i}_q$ possibilities for the choice of $V_\omega$ and we get thus
the identity
\begin{align*}
\sum_{F\in \mathcal{WF}(\mathbb F_q^d)}t^{w(F)}&=\sum_{i=0}^d
t^i{d\choose i}_q\quad \sum_{F\in \mathcal{WF}(\mathbb F_q^i)}t^{w(F)}
\end{align*}
equivalent to
\begin{align*}
(1-t^d)\sum_{F\in \mathcal{WF}(\mathbb F_q^d)}t^{w(F)}&=\sum_{i=0}^{d-1}
t^i{d\choose i}_q\quad \sum_{F\in \mathcal{WF}(\mathbb F_q^i)}t^{w(F)}
\end{align*}
which is the generating series of all weighted non-maximal flags.

Applying Theorem \ref{mainthm} on both sides we get 
\begin{align*}
M_d\prod_{j=1}^{d-1}\frac{1}{1-t^j}&=\sum_{i=0}^{d-1}
t^i{d\choose i}_qM_i\prod_{j=1}^i\frac{1}{1-t^j}
\end{align*}
which yields the result multiplying both sides with
$\prod_{j=1}^{d-1} (1-t^j)$.
\end{proof}

\begin{proof}[Proof of Proposition \ref{propevqeqoneofMd}]
Equality holds trivially for $d=1$. Induction yields
for the evaluation $A_d$ of $M_d$ at $q=1$
\begin{align*}
A_d&=\sum_{i=0}^{d-1}t^i\left(\prod_{j=i+1}^{d-1}(1-t^j)\right){d\choose i}\prod_{j=1}^i\frac{1-t^j}{1-t}\\
&=\prod_{j=1}^{d-1}(1-t^j)\left(\left(\sum_{i=0}^d{d\choose i}\frac{t^i}{(1-t)^i}\right)-\frac{t^d}{(1-t)^d}\right)\\
&=\prod_{j=1}^{d-1}(1-t^j)\left(\left(1+\frac{t}{1-t}\right)^d-\frac{t^d}{(1-t)^d}\right)\\
&=\prod_{j=1}^d\frac{1-t^j}{1-t}\ .
\end{align*}
\end{proof}

%\begin{rem} Limit series for finitary permutation group. {\bf Expand into Section}\end{rem}
%{\bf REWORK following remark!!!!!!!!!!!!!!!!!:}
%\begin{rem} Remark on $\lim_{d\rightarrow \infty}M_d$: 
%Combinatorial types (consider only dimensions) of all partitions
%enumerated according to  weights converge to the generating sequence
%A9 enumerating partitions with distinct parts. (More precisely,
%such types of $\mathbb F^d$ are given by partitions with distinct parts
%strictly smaller than $d$).
%
%Combinatorial types of all weighted partitions of $\mathbb F^d$ are
%enumerated (accordingly to  weights) by partitions into parts 
%$\leq d$.
%\end{rem}

\section{Canonical presentations of weighted flags}

We associate to a flag $F=(V_1\subset V_2\subset\dots\subset V_k)$ of $\mathbb F^d$
a uniquely defined \emph{canonical basis} 
$f_1,\dots,f_d$ of $\mathbb F^d$ such that 
$V_i$ is spanned by $f_1,\dots,f_{\dim(V_i)}$ for all $i$.

The \emph{canonical presentation}
of a weighted flag is given by its canonical basis
$f_1,\dots,f_d$ together with its conjugate weight-partition
$\mu(f_1),\dots,\mu(f_d)$.
Canonical presentations are given algorithmically and thus uniquely
determined.

The canonical basis $f_1,\dots,f_k$ is constructed algorithmically as follows:

If $V_k$ is a strict subspace of $\mathbb F^d$, we increase the flag
by adding a last space $V_{k+1}=\mathbb F^d$ with degenerate multiplicity 
$\mu(v)=0$ on elements in $V_{k+1}\setminus V_k$. 

The \emph{length} of a non-zero element $v=(x_1,\dots,x_j,0,\dots,0)$ 
in $\mathbb F^d$ is the largest integer $j$ corresponding to a non-zero 
coordinate $x_j\in\mathbb F\setminus\{0\}$.

Let $f_1$ be the unique element of smallest length 
and last coordinate $1$ in $V_1$. Set $\lambda(1)=\mathop{length}(f_1)$.
If $V_1$ is of dimension larger than $1$, define $f_2$ as the unique shortest
non-zero element of $V_1$ with last coordinate 
$1$ and coordinate $x_{\lambda(1)}=0$. 
If $V_1$ is one-dimensional, define $f_2$ similarly using $V_2$.
Set $\lambda(2)=\mathop{length}(f_2)$.

More generally, suppose $f_1,\dots,f_{i-1}$ with $i< d$ already 
constructed. Let $j$ be the smallest index such that $V_j$ is not
spanned by $f_1,\dots,f_{i-1}$.
The vector $f_i$ is defined as the unique shortest element in $\mathbb F^d$ of $V_j\setminus (\mathbb Ff_1+\dots+\mathbb Ff_{i-1})$
with last coordinate $1$ and with coordinates $x_{\lambda(1)}=
x_{\lambda(2)}=\dots=x_{\lambda(i-1)}=0$. We set $\lambda(i)=\mathop{length}(f_i)$.

This construction yields a basis $f_1,\dots,f_d$ of $\mathbb F^d$ such that
$V_i$ is spanned by $f_1,\dots,f_{\dim(V_i)}$ for all $i$.
We call $f_1,\dots,f_d$ the \emph{canonical basis}.

The sequence $\lambda(f_1),\lambda(f_2),\dots$ of lengths of the canonical
basis defines a permutation $i\longmapsto \lambda(f_i)$ of $\mathcal S_d$.
We $\lambda$ the \emph{length-permutation}.
The length-permutation is uniquely defined in 
terms of the canonical basis.

Observe that many flags share a common canonical basis, see
Remark \ref{remnuberofrefnmts} for more.

We will see in Proposition \ref{propweightedflagstandardflag} 
that weighted flags corresponding to a given standard basis of $\mathbb F^d$
are in one-to-one correspondence with 
partitions having at most $d$ parts.
Each such flag contributes one monomial to the factor
$\prod_{j=1}^d(1-t^j)^{-1}$ appearing in Formula(\ref{formofmainthm})
of Theorem \ref{mainthm}. 
The constant term $1$
corresponds to the standard flag, introduced and studied in the
next Section.

\begin{rem}
The construction of the canonical basis is equivalent to 
the Bruhat decomposition $BWB=\mathrm{SL}_d(\mathbb F)$
of special linear groups.
The length-permutation $\lambda$ belongs to the Weyl group 
$\mathcal S_d$ of $\mathrm{SL}_d(\mathbb F)$ and indexes the Schubert cell
$B\lambda B$ containing the matrix with rows $f_1,\dots, f_{d-1},
(-1)^{\mathrm{sign}(\lambda)}f_d$ (with $\mathrm{sign}(\lambda)$
denoting the signature of $\lambda$).
\end{rem}

\section{Standard flags}

A flag $F=(
V_1\subset \dots\subset V_k\subset \mathbb F^d)$ of $\mathbb F^d$ consisting 
of $k$ subspaces
is a \emph{standard flag} if $F$ is  a nonmaximal flag (contained in a
strict subspace $V_k$ of $\mathbb F^d$) associated to a length-permutation
with $k$ descents. (Recall that a descent is an index $i<d$ such that 
$\sigma(i)>\sigma(i+1)$.)
Since $\lambda(i)<\lambda(i+1)$ if $i\not\in\{\dim(V_1),\dots,\dim(V_k)\}$, 
standard flags involve the minimal
number of subspaces determining a given length-permutation.
Every canonical basis defines a unique standard flag 
defined by $\dim(V_i)=j$ if $j$ is the $i$-th descent of the 
corresponding length-permutation. 
Any abstract flag $F=U_1\subset \dots\subset U_l$ 
with canonical 
basis $f_1,\dots,f_d$ is a refinement of the standard flag $V_1\subset
\dots\subset V_k$ defined
by $f_1,\dots,f_d$ in the following sense: there exists a strictly 
increasing sequence $1\leq j_1<\dots<j_k\leq l$ such that $V_i=U_{j_i}$.
We denote by $st(F)$ the standard flag
associated to the canonical basis of a flag $F$.
In particular, every standard flag can be refined 
to a unique complete flag $(U_1\subset \dots\subset U_d)$
by definining $V_i$ as the span of 
the first $i$ elements $f_1,\dots,f_i$ of its canonical basis.
Standard flags are thus in one-to-one correspondence with 
complete flags. (Observe however that the definition of complete
flags is independent of the choice of a basis. Standard flags 
are defined in $\mathbb F^d$ with respect to the natural basis 
of $\mathbb F^d$. The map $F\longmapsto st(F)$ associating 
to a complete flag $F$ the standard flag $st(F)$ is thus 
only defined for complete flags of $\mathbb F^d$.)

\begin{rem}\label{remnuberofrefnmts}
A canonical basis $f_1,\dots,f_d$ with length permutation $\lambda$
having $k$ descents is the canonical basis of exactly $2^{d-k}$ 
different abstract flags: 
Each quotient space $V_i/V_{i-1}$ of the standard flag 
can indeed be refined in
$2^{\dim(V_i)-\dim(V_{i-1})-1}$ different ways (corresponding to 
the $2^{\dim(V_i)-\dim(V_{i-1})-1}$ different compositions of the
natural integer $\dim(V_i)-\dim(V_{i-1})$) into abstract flags 
having the same canonical basis. 
\end{rem}

The \emph{standard weight} of a standard flag $F=(V_1\subset \dots\subset V_k)$ 
is given by 
$$w_{\min}(F)=\sum_{i=1}^k\dim(V_i)\ .$$
It corresponds to the smallest 
possible weight-sequence 
$w(V_1)=\dots=w(V_k)=1$. The associated
weight-multiplicities are given by $\mu(f_j)=k+1-i$ if 
$\dim(V_{i-1})<j\leq \dim(V_i)$ and by $\mu(f_j)=0$ for $j>\dim(V_k)$.

The  standard weight of a standard flag
has the following combinatorial characterization:

\begin{prop}\label{propstandweight} The  standard weight $w_{\min}(F)$ of a standard flag
$F=(V_1\subset \dots\subset V_k)$ 
with length permutation $\lambda\in\mathcal S_d$
is given by the Major index
$$\mathop{maj}(\lambda)=\sum_{i,\lambda(i)>\lambda(i+1)}i$$
of $\lambda$.
\end{prop}

%\begin{proof}[]
We leave the obvious proof to the reader.
%\end{proof}

We illustrate the notion of a standard flag and of its minimal weight
by an example. We consider a canonical basis $f_1,\dots,f_9$ of $\mathbb F^9$
with length permutation 
$(\lambda(1),\lambda(2),\dots,\lambda(9))=(6,3,8,1,4,9,7,2,5)$
the permutation illustrating the Wikipedia entry ``Permutation''
at the time of writing. Dimensions $\dim(V_i)-\dim(V_{i-1})$ of 
quotient-spaces $V_i/V_{i-1}$ correspond to (non-final) maximal ascending runs
(maximal sets of consecutive integers on which $\lambda$ is increasing).
The ascending runs of $\lambda$ are
$$\begin{array}{c|c}
\hbox{run}&\lambda(\hbox{run})\\
\hline
1&6\\
2,3&3,8\\
4,5,6&1,4,9\\
7&7\\
8,9&2,5
\end{array}$$
and correspond to the composition
$1+2+3+1+2$ of $d=9$. Since the Major index
$\sum_{i,\sigma(i)>\sigma(i+1)}i$ of $\lambda$ is the sum 
of preimages (or indices) of 
maximal elements in non-final ascending runs (maximal sets of consecutive 
integers on which the permutation is increasing), 
the Major index of our example equals
$1+3+6+7=17$.
An associated standard flag is given by 
\begin{align*}
V_1&=\mathbb F f_1,\\
V_2&=\mathbb F f_1+\mathbb F f_2+\mathbb F f_3,\\
V_3&=\mathbb F f_1+\dots+\mathbb F f_6,\\
V_4&=\mathbb F f_1+\dots+\mathbb F f_7
\end{align*}
for a suitable basis $f_1,\dots,f_9$
with weight-multiplicities $\mu(f_1)=4,\mu(f_2)=\mu(f_3)=3,\mu(f_4)=\mu(f_5)=\mu(f_6)=2,\mu(f_7)=1,\mu(f_8)=\mu(f_9)=0$. 
The standard weight of this standard flag is given by
$4\cdot 1+3\cdot 2+2\cdot 3+1\cdot 1=17$ and is thus equal to the Major index 
of $\sigma$.

The observation that the 
standard weight of a standard flag is the smallest possible  weight
among all weighted flags sharing a given 
canonical basis can be refined into the following result:

\begin{prop}\label{propweightedflagstandardflag} Let $F=(U_1\subset \dots\subset U_l)$
be a weighted flag with associated standard flag $st(F)=(
V_1\subset \dots\subset V_k)$.
There exists a partition $\alpha_1,\dots,\alpha_d$
of $w_F(F)-w_{\min}(st(F))$ 
such that the weight-multiplicities $\mu_F(f_i)$ for $F$ are given by 
$$\mu_F(f_i)=\mu_{\min}(f_i)+\alpha_i$$
with $\mu_{\min}(f_i)$  defining the standard weight-multiplicities
of the associated standard flag $st(F)$.

The  weight of $F$ is given by
$$w_F(F)=w_{\min}(F)+\sum_{i=1}^d \alpha_i\ .$$
\end{prop}

Proposition \ref{propweightedflagstandardflag} implies the following result:

\begin{cor}\label{corpartitionstandard} We have
$$\sum_{F\in{st}^{-1}(F_{st})}t^{w(F)}=
t^{w(F_{st})}\prod_{i=1}^d\frac{1}{1-t^i}$$
where $w(F_{st})$ is the standard
weight of a standard flag $F_{st}$ and where ${st}^{-1}(F_{st})$
is the set of all weighted flags with associated standard flag $F_{st}$.
\end{cor}

\begin{proof} This follows from the easy observation 
that any partition $\alpha_1,\dots,\alpha_d$ defines
a sequence of weight-multiplicities $w(f_i)=\alpha_i+w_{\min}(f_i)$
on the standard basis $f_1,\dots,f_d$.
\end{proof}

\begin{proof}[Proof of Proposition \ref{propweightedflagstandardflag}]
We have to show that the sequence $\alpha_1,\dots,\alpha_d$
defined by $\alpha_i=\mu_F(f_i)-\mu_{\min}(f_i)$ has decreasing non-negative 
values.
Suppose first that we have $\alpha_i<\alpha_{i+1}$.
Since $\mu_{\min}(f_i)-\mu_{\min}(f_{i+1})\leq 1$ and since $\mu_F(f_i)$ is
decreasing, the inequality $\alpha_i<\alpha_{i+1}$ is only possible
if $\mu_F(f_i)=\mu_F(f_{i+1})$ and $\mu_{\min}(f_i)=1+\mu_{\min}(f_{i+1})$.
This contradicts the fact that $F$ contains all subspaces of 
the associated standard flag $st(F)$. 

Non-negativity of $\alpha_i$ follows from $\alpha_1\geq \alpha_2\geq  
\dots\geq\alpha_d
=\mu_F(f_d)-\mu_{\min}(f_d)=\mu_F(f_d)-0\geq 0.$
\end{proof}

\section{Canonical bases with a given length-permutation}

\begin{prop}\label{propcanbaseslambda} 
The set $\mathcal B(\lambda)$
of all canonical bases of $\mathbb F^d$ with
length-permutation $\lambda\in\mathcal S_d$ is an
affine vector-space of dimension
$$\mathop{inv}(\lambda)=\sum_{i<j,\lambda(i)>\lambda(j)}1\ .$$
\end{prop}

The crucial ingredients for proving Theorem \ref{mainthm} are
Corollary \ref{corpartitionstandard} and the following result:

\begin{cor}\label{corstflags} Denoting by $\mathcal{F}_{st}(\mathbb F^d)$
the set of all standard flags of $\mathbb F^d$ we have
$$\sum_{F\in\mathcal{F}_{st}(\mathbb F_q^d)}t^{w_{st}(F)}=
\sum_{\lambda\in \mathcal S^d}t^{\mathop{maj}(\lambda)}q^{\mathop{inv}(\lambda)}$$
where $w_{st}(F)$ denotes the standard weight of a standard flag $F$
and where $\mathop{maj}$ and 
$\mathop{inv}$ denote the Major index and the number of inversions 
of a permutation $\lambda$ in 
$\mathcal S_d$.
\end{cor}

\begin{proof}[Proof of Corollary \ref{corstflags}]
We apply Propositions \ref{propstandweight} and
\ref{propcanbaseslambda} to the 
sum over all elements of $\mathcal S_d$.
\end{proof}

\begin{proof}[Proof of Proposition \ref{propcanbaseslambda}] 
We construct the set $\mathcal B(\lambda)$ of all canonical 
bases of $\mathbb F^d$ with length-permutation
$\lambda$ in $\mathcal S_d$.

The first element $f_1=(x_{1,1},x_{1,2},\dots,x_{1,\lambda(1)-1},1,0,\dots,0)$ 
has $\lambda(1)-1$ arbitrary coefficients
followed by a coefficient $x_{1,\lambda(1)}=1$. Coefficients with indices
larger than $\lambda(1)$ are $0$.
More generally, the $i$-th basis element $f_i$ of a canonical basis
has prescribed coeffients $x_{i,\lambda(j)}=0$ for $j=1,\dots,i-1$,
$x_{i,\lambda(i)}=1$, $x_{i,j}=0$ for $j>\lambda(i)$. All remaining
coefficients are free. Such coefficients are in bijection with $j>i$
such that $\lambda(j)<\lambda(i)$. Their number is thus the number
of inversions 
$i<j, \ \lambda(i)>\lambda(j)$ involving $i$ as their smaller argument.
This shows that there are $q^{\mathop{inv}(\lambda)}$ canonical bases of
$\mathbb F_q^d$ with length-permutation a given element $\lambda$
in $\mathcal S_d$.
\end{proof}

The \emph{Rothe-diagram} of $\lambda$ visualizes the
nature of the coefficients occuring in a canonical basis of 
the set $\mathcal B(\lambda)$ constructed while 
proving Proposition \ref{propcanbaseslambda}: Bullets $\bullet$ 
represent coefficients $1$, crosses $\times$
represent free coefficients and blanks represent coefficients
which are necessarily $0$.

For example, the Rothe diagram of permutation $(\lambda(1),\lambda(2),\dots,\lambda(9))=(6,3,8,1,4,9,7,2,5)$
of $\mathcal S_9$ 
is given by
$$\begin{array}{c||c|c|c|c|c|c|c|c|c|}
i\backslash \lambda(i)&1&2&3&4&5&6&7&8&9\\
\hline\hline
1&\times&\times&\times&\times&\times&\bullet&&&\\
\hline
2&\times&\times&\bullet&&&&&&\\
\hline
3&\times&\times&&\times&\times&&\times&\bullet&\\
\hline
4&\bullet&&&&&&&&\\
\hline
5&&\times&&\bullet&&&&&\\
\hline
6&&\times&&&\times&&\times&&\bullet\\
\hline
7&&\times&&&\times&&\bullet&&\\
\hline
8&&\bullet&&&&&&&\\
\hline
9&&&&&\bullet&&&&
\end{array},$$
cf. the section \lq\lq Numbering permutations\rq\rq of the Wikipedia
entry \lq\lq Permutation\rq\rq in \cite{Wi}.
Its 18 crosses correspond to the $18$ inversions giving the length of 
$\lambda$. Each cross has indeed exactly one bullet at its right
and one bullet below it. The lines $i<j$ indexing these two bullets
define an inversion $\lambda(i)>\lambda(j)$ of $\lambda$.

\section{Type C: Symplectic flags}

A symplectic space $V$ is always of even dimension $2d$ and has a 
basis $e_1,\dots,e_d,f_1,\dots,f_d$ such that $\omega(e_i,f_i)=-\omega(f_i,e_i)=1$ and $\omega(e_i,e_j)=\omega(f_i,f_j)=\omega(e_i,f_j)=0$ if $i\not= j$.
We denote a symplectic space over a field $\mathbb F$ 
henceforth by $(\mathbb F^{2d},\omega)$ or $\mathbb F^{2d}$ for short.
A subspace $W$ of $V$ is \emph{isotropic} if the restriction of $\omega$
to $W\times W$ is identically zero. \emph{Lagrangians} in 
$(\mathbb F^{2d},\omega)$ are maximally
isotropic subspaces and are of dimension $d$.

A symplectic flag of $(\mathbb F^{2d},\omega)$ is 
a flag $V_1\subset V_2\subset \dots\subset V_k$ of $\mathbb F^{2d}$ which is 
contained in a Lagrangian of $\mathbb F^{2d}$.
The symplectic form $\omega$ restricts thus 
to $0$ on $V_k\times V_k$ and $V_k$ is of dimension at most $d$.
A symplectic flag is \emph{maximal} if $\dim(V_k)=d$ and
\emph{complete} if $k=d$. Complete
symplectic flags are maximal. 

Weighted flags and weights of flags are defined in 
the obvious way.
We denote by $\mathcal{WF}(\mathbb F^{2d},\omega)$ 
the set of all weighted symplectic
flags of the symplectic space $(\mathbb F^{2d},\omega)$.

\subsection{Proof of Corollary \ref{corformulaMdpm}}

The proof of Corollary \ref{corformulaMdpm} is essentially identical 
to the proof of Corollary \ref{correc}. The main ingredient is the following
well-known result which can be seen as an analogue of $q-$binomial
coefficients:

\begin{lem}\label{lemistropBC} (i) Given an odd prime power $q$, the symplectic space 
$(\mathbb F_q^{2d},\omega)$ and a non-degenerate orthogonal space of
dimension $2d+1$ over $\mathbb F_q$ have both $q^{2d}$ isotropic elements.

(ii) Given an odd prime power $q$, the symplectic space 
$(\mathbb F_q^{2d},\omega)$ and a non-degenerate orthogonal space of
dimension $2d+1$ over $\mathbb F_q$ have both
\begin{align}\label{formisotrinsympl}
\prod_{j=0}^{k-1}\frac{q^{2d-j}-q^j}{q^k-q^j}
&=\prod_{j=0}^{k-1}\frac{1-q^{2d-2j}}{1-q^{k-j}}
\end{align}
isotropic subspaces of dimension $k$.
\end{lem}

We leave the proof to the reader.

\begin{proof}[Proof of Corollary \ref{corformulaMdpm}] 
%The number of $k$-dimensional isotropic subspaces in 
%$(\mathbb F_q^{2d},\omega)$ is given by 
%\begin{align}\label{formisotrinsympl}
%\prod_{j=0}^{k-1}\frac{q^{2d-j}-q^j}{q^k-q^j}
%&=\prod_{j=0}^{k-1}\frac{1-q^{2d-2j}}{1-q^{k-j}}
%\end{align}
%for $k\in \{0,\dots,d\}$.
Assertion (ii) of Lemma \ref{lemistropBC} and Theorem \ref{mainthm} imply
$$\sum_{F\in\mathcal{WF}(\mathbb F_q^{2d},\omega)}t^{w(F)}=
\sum_{k=0}^dt^k\left(\prod_{j=0}^{k-1}\frac{q^{2d-2j}-1}{q^{k-j}-1}\right)
M_k\prod_{j=1}^k\frac{1}{1-t^j}\ .$$
Use of Theorem \ref{thmsympl}
and multiplication by $\prod_{j=1}^d(1-t^j)$ end the proof. 
\end{proof}

\subsection{Small values and a few properties for $M_d^\pm$}

The first few polynomials $M_d^\pm$ are as follows: 
$M^\pm_1=1+qt$, coefficients of $M^\pm_2,M^\pm_3$ are given by
$$
{\begin{array}{c|ccccccc}
&1&q&q^2&q^3&q^4\\
\hline
1&1\\
t&&1&1&1\\
t^2&&1&1&1\\
t^3&&&&&1\\
%\hline
%&1&2&2&2&1
\end{array}}
\quad
{\begin{array}{c|cccccccccc}
&1&q&q^2&q^3&q^4&q^5&q^6&q^7&q^8&q^9\\
\hline
1&1\\
t&&1&1&1&1&1\\
t^2&&1&2&2&2&2&1&1\\
t^3&&1&1&3&2&2&3&1&1\\
t^4&&&1&1&2&2&2&2&1\\
t^5&&&&&1&1&1&1&1\\
t^6&&&&&&&&&&1\\
%\hline
%&1&3&5&7&8&8&7&5&3&1
\end{array}}
$$
and coefficients of $M^\pm_4$ are given by
$${\begin{array}{c|ccccccccccccccccc}
&0&1&2&3&4&5&6&7&8&9&10&11&12&13&14&15&16\\
\hline
1&1\\
t&&1&1&1&1&1&1&1\\
t^2&&1&2&2&3&3&3&3&2&2&1&1\\
t^3&&1&2&4&4&6&6&6&6&5&4&2&2\\
t^4&&1&2&4&6&7&8&9&9&8&7&5&3&2&1\\
t^5&&&1&3&5&7&9&10&12&10&9&7&5&3&1\\
t^6&&&1&2&3&5&7&8&9&9&8&7&6&4&2&1\\
t^7&&&&&2&2&4&5&6&6&6&6&4&4&2&1\\
t^8&&&&&&1&1&2&2&3&3&3&3&2&2&1\\
t^9&&&&&&&&&&1&1&1&1&1&1&1\\
t^{10}&&&&&&&&&&&&&&&&&1\\
%\hline
%&1&4&9&16&24&32&39&44&46&44&39&32&24&16&9&4&1
\end{array}}$$
(with columns yielding the coefficients of 
$1,q,\dots,q^{16}$).
Maximal degrees in $q$ and $t$ of  $M_d^\pm$
are due to the Coxeter element $c(i)=-i,\ i=1,\dots,d$
(which is the unique non-trivial central element in $\mathcal S_d^\pm$)
of maximal length ${d\choose 2}+{d+1\choose 2}=d^2$ 
and maximal flag-Major index ${d\choose 2}+d={d+1\choose 2}$
contributing the monomial $q^{d^2}t^{d+1\choose 2}$
of leading degree in $q$ and $t$ to $M_d^\pm$.
The easy identities $l(\sigma)+l(c\circ \sigma)=d^2$ and
$\mathop{Wmaj}(\sigma)+\mathop{Wmaj}(c\circ \sigma)={d+1\choose 2}$
imply the symmetry $t^{d\choose 2}q^{d^2}M^\pm_d(1/q,1/t)=M^\pm_d(q,t)$,
obvious in the above examples.

The polynomial $M_d^\pm$ evaluates to $2^dd!$ at $q=t=1$.

Corollary \ref{corformulaMdpm}
and Formula (\ref{formulalengthfactorSd}) yield the well-known
factorization 
$$M_d^\pm(q,1)=\sum_{\sigma\in\mathcal S_d^\pm}
q^{l(\sigma)}=\prod_{j=1}^d\frac{1-q^{2j}}{1-q}$$
(analogous to the factorization for $\mathcal S_d$ 
given by (\ref{formulalengthfactorSd}))
for the generating polynomial of lengths in the 
hyperoctahedral group $\mathcal S_d^\pm$ (obtained by evaluating
$M_d^\pm$ at $t=1$). 

A similar computation (again using Corollary \ref{corformulaMdpm}
and Formula (\ref{formulalengthfactorSd}))
yields the factorization
$$M^\pm_d(1,t)=(1+t)^dM_d(1,t)=(t+1)^d\prod_{j=1}^d\frac{t^j-1}{t-1}$$
at $q=1$.

The easy congruence
$${d\choose i}_q\equiv \prod_{j=0}^{i-1}\frac{1-q^{2d-2j}}{1-q^{i-j}}
\pmod{q^{d+1-i}}$$
implies that $M_d$ and $M_d^{\pm}$ have identical coefficients of 
total degree $\deg_t+\deg_q$ at most $d$.
We have thus
$\lim_{d\rightarrow\infty}M_d(q,q)=\lim_{d\rightarrow\infty}M_d^\pm(q,q)
\in \mathbb Z[[q]]$ for coefficient-wise convergency.

\section{The length-function of $\mathcal S_d^\pm$}

We remind the reader that a pair $i<j$ such that
$\sigma(i)>_\pm \sigma(j)$ (with $<_\pm$ defined by Formula (\ref{pmorderonZ}))
is an inversion of an element $\sigma$ in the hyperoctahedral group $S_d^\pm$.
(Readers annoyed by the order relation $<_\pm$ can replace
the condition $\sigma(i)>_\pm \sigma(j)$ by the equivalent condition 
$\sigma(i)\sigma(j)(\sigma(i)-\sigma(j))>0$.)

Inversions $i<j$ can be classified by their sign-pattern into three types:
$\sigma(i)>\sigma(j)>0,\sigma(i)<0<\sigma(j)$ and $0> \sigma(i)>\sigma(j)$
depicted graphically by
$$\begin{array}{cc}
\bullet&\\
&\bullet\\
\hline
\end{array}\quad,\qquad 
\begin{array}{cc}
&\bullet\\
\hline
\bullet&\\
\end{array}\hbox{ and }
\begin{array}{cc}
\hline
\bullet&\\
&\bullet\\
\end{array}$$
using hopefully self-explanatory notations.

Non-inversions $i<j$ are similarly classified into
$0<\sigma(i)<\sigma(j),\sigma(i)>0>\sigma(j)$ and $\sigma(i)<\sigma(j)<0$
represented by
$$\begin{array}{cc}
&\bullet\\
\bullet&\\
\hline
\end{array}\quad,\qquad 
\begin{array}{cc}
\bullet&\\
\hline
&\bullet\\
\end{array}\hbox{ and }
\begin{array}{cc}
\hline
&\bullet\\
\bullet&\\
\end{array}$$

We call the contribution
$\sum_{0<i,\sigma(i)<0}(d+1+\sigma(i))$ to $l(\sigma)$ (given by Formula
(\ref{defsignedlength})) the \emph{sign-part} of the length.

We denote the ordinary transpositions $(i,i+1)$ of $M_d^\pm$ 
by $s_i$ for $i=1,\dots,d-1$ and we denote by
$s_d$ the sign change of the last coordinate (exchanging $d$ and
$-d$).

\begin{proof}[Proof of Proposition \ref{propsignedlength}] 
For the sake of concision, we denote by $l$ the function of 
Proposition \ref{propsignedlength} defined by 
(\ref{defsignedlength}).
The length $0$ of the identity permutation $\sigma(i)=i$ for $i\in\{1,\dots,d\}$
is obviously given by $l(\sigma)=0$.

If $i<d$, the map $\sigma\longmapsto \sigma\circ s_i$ does not affect 
the sign-part of $\sigma$ in $\mathcal S_d^\pm$. Moreover,
since $i<i+1$ is an inversion of $\sigma$ if and only if it is a non-inversion
of $\sigma\circ s_i$ and vice-versa, 
the number of inversions of 
$\sigma\circ s_i $ and of $\sigma$ differ 
exactly by $1$ for $i=1,\dots,d-1$. 
This implies 
$\vert l(\sigma)-l(\sigma\circ s_i)\vert=1$ for $i<d$.

We prove now that this holds also for $i=d$, i.e. we have
$\vert l(\sigma)-l(\sigma\circ s_d)\vert=1$:
Up to replacing $\sigma$ with $\sigma\circ s_d$ we can 
suppose $\sigma(d)=a>0$.
We depict $\sigma$ schematically by the following representation
$$\begin{array}{r||cccc|c}
\sigma(j)&i_1&i_2&i_3&i_4&d\\
\hline\hline
&&&&\bullet_4&\\
a&&&&&\bullet\\
&&&\bullet_3&&\\
\hline
&&\bullet_2&&&\\
-a&&&&&\circ\\
&\bullet_1&&&&\\
\end{array}$$
(the horizontal line represents $0$, the vertical line 
separates the last index $d$ from previous ones,
the last value of $\sigma$, respectively of $\tilde \sigma=\sigma\circ s_d$, is represented by $\bullet$, respectively $\circ$).
We denote by $i_j$ indices taking values depicted by $\bullet_j$, i.e.
$$-d\leq\sigma(i_1)<-a<\sigma(i_2)<0<\sigma(i_3)<a<
\sigma(i_4)\leq d\ .$$
The following Table depicts the status with respect to inversions
(Yes for inversions, No for non-inversions)
of $\sigma$ and $\tilde \sigma=\sigma\circ s_d$
for $i_j<d$:
$$\begin{array}{c||c||c}
j&\sigma(i_j)>_\pm \sigma(d)&
\tilde\sigma(i_j)>_\pm \tilde\sigma(d)\\
\hline
1&\mathrm{Yes}&\mathrm{No}\\
2&\mathrm{Yes}&\mathrm{Yes}\\
3&\mathrm{No}&\mathrm{No}\\
4&\mathrm{Yes}&\mathrm{No}\\
\end{array}$$
Setting
$$\nu_j=\sharp\{i<d\vert i\text{ is of type }i_j\}$$
(where "type" means represented by $\bullet_j$)
we get now
$$\begin{array}{rcl}
l(\sigma\circ s_d)-l(\sigma)&=&d+1-a-\nu_1-\nu_4\\
&=&d+1-a-(d-a)\\
&=&1
\end{array}$$
using the trivial identity $\nu_1+\nu_4=d-a$. 
This shows the equality
\begin{align}\label{formsigmasdsigma}
\vert l(\sigma)-l(\sigma\circ s_d)\vert=1.
\end{align}

So far we have proven that $l(\sigma)$, given by Formula 
(\ref{defsignedlength})
of Proposition \ref{propsignedlength}, is at least equal 
to the length of $\sigma$ in terms of 
the generators $s_1,\dots,s_d$. Indeed, since composition with a 
generator changes the value of $l$ exactly by $1$ and since $l(\sigma)=0$
if $\sigma$ is the identity, at least $l(\sigma)$ generators are 
necessary for writing an arbitrary element $\sigma$ of $S_d^\pm$.

Equality holds if we show that 
for each $\sigma$ with strictly positive $l(\sigma)$ there exists 
a generator $s_i$ such that $l(\sigma\circ s_i)=l(\sigma)-1$.
We consider thus an arbitrary signed permutation $\sigma$ of 
$\mathcal S_d^\pm$.

If $\sigma(d)<0$, then 
$l(\sigma\circ s_d)=l(\sigma)-1$ as can be seen after interchanging
$\sigma,\sigma\circ s_d$ and applying the identity
(\ref{formsigmasdsigma}).

If $\sigma(d)>0$ then $\sigma$ is either the identity or there exists
a largest integer $i<d$ such that either $\sigma(i)<0<\sigma(i+1)$ or
$\sigma(i)>\sigma(i+1)>0$. In both cases, $i<i+1$ defines an inversion
of $\sigma$ but not of $\sigma\circ s_i$ and we have thus $l(\sigma\circ s_i)=
l(\sigma)-1$.
\end{proof}

\begin{rem}
The proof of Proposition \ref{propsignedlength} is algorithmic. 
Given an element $\sigma\not=\mathop{id}$ of $\mathcal S_d^\pm$,
we choose an index $i_1$ (for example as large as possible) such that 
%$s_{i_1}\circ \sigma$ 
$\sigma\circ s_{i_1}$ 
is of shorter length than $\sigma$ and we iterate
until arriving at the identity
$s\circ s_{i_1}\circ s_{i_2}\circ\cdots \circ s_{i_l}$
(with $l=l(\sigma)$ denoting the length of $\sigma$ given by 
Formula (\ref{defsignedlength}) in Proposition \ref{propsignedlength}).
This leads to a shortest expression
$$\sigma=s_{i_l}\circ s_{i_{l-1}}\circ \cdots\circ s_{i_1}$$
of $\sigma$ using $l=l(\sigma)$ (not necessarily different) generators $s_{i_j}$ 
belonging to the set $\{s_1=(1,2),s_2=(2,3),\dots,s_{d-1}=(d-1,d),s_d=(d,-d)\}$.
Every shortest expression with respect to the Coxeter generators $s_1,\dots,
s_d$ is obtained in this way.

Denoting a signed partition $\sigma$
by $(\sigma(1),\dots,\sigma(d))$, we get for example for $\sigma=(-2,-3,1)$
(of length $6$, it has $3$ inversions and a sign-part of $(4-2)+(4-3)=3$):
$$\begin{array}{r|c|c}
\sigma&(-2,-3,1)&6\\
\sigma\circ s_2&(-2,1,-3)&5\\
\sigma\circ s_2\circ s_3&(-2,1,3)&4\\
\sigma\circ s_2\circ s_3\circ s_1&(1,-2,3)&3\\
\sigma\circ s_2\circ s_3\circ s_1\circ s_2&(1,3,-2)&2\\
\sigma\circ s_2\circ s_3\circ s_1\circ s_2\circ s_3&(1,3,2)&1\\
\sigma\circ s_2\circ s_3\circ s_1\circ s_2\circ s_3\circ s_2&(1,2,3)&0
\end{array}$$
(with the last column indicating lengths) when choosing always largest possible
indices.
This gives for $\sigma=(-2,-3,1)$ the expression
$$\sigma=(s_2\circ s_3\circ s_1\circ s_2\circ s_3\circ s_2)^{-1}=
s_2\circ s_3\circ s_2\circ s_1\circ s_3\circ s_2$$
in terms of the generators $s_1=(2,1,3),s_2=(1,3,2),s_3=(1,2,-3)$.
\end{rem}

\section{The canonical half-basis of a symplectic flag}\label{secthalfbases}

We embed a flag $F=(V_1\subset \dots\subset V_k)$ of $\mathbb F^{2d}
=(\mathbb F^{2d},\omega)$ 
canonically in a Lagrangian (maximal isotropic subspace) $L(V_k)$
depending only on $V_k$ and we associate to such a flag a length-permutation
$\sigma\in\mathcal S_d^\pm$ and a canonical basis $f_1,\dots,f_d$ of $L(V_k)$
such that $f_1,\dots,f_{\dim(V_i)}$ span $V_i$ for $i=1,\dots,k$.

We identify the symplectic space $\mathbb F^{2d}$ with 
$\oplus_{i=1}^d(\mathbb Fb_i+\mathbb F b_{-i}),i\in\{1,\dots,d\}$,
endowed with the symplectic form
$\omega(b_i,b_j)=0$ if $i+j\not=0$,
$\omega(b_i,b_{-i})=-\omega(b_{-i},b_i)=1$ for $i=1,\dots,d$.

We order indices of
coordinates $x_i$ of an element $\sum_{i=1}^d(x_ib_i+x_{-i}b_{-i})$ in 
$\mathbb F^{2d}$ always increasingly with respect to the order relation
$<_\pm$ giving the order relation
$$1<_\pm 2<_\pm 3<_\pm \dots <_\pm -3<_\pm -2<_\pm -1$$
of (\ref{pmorderonZ}). We write thus elements of $\mathbb F^{2d}$
in the form
\begin{align}\label{symplindices}
(x_1,x_2,\dots,x_{d-1},x_d,x_{-d},x_{-d+1},\dots,x_{-2},x_{-1})\ .
\end{align}
We use this order also on indices of the standard basis of 
$(\mathbb F^{2d},\omega)$ writing it in the form 
$$b_1,b_2,b_3,\dots,b_{d-1},b_d,b_{-d},b_{1-d},b_{2-d},\dots,b_{-2},b_{-1}\ .$$

We call the position of the index $i$ with respect to the
order $<_\pm$ the  \emph{length} of a basis element $b_i$.
More precisely, the length $\lambda(b_i)$ of $b_i\in\mathbb F^{2d}$ is given by
$$\lambda(b_i)=\left\lbrace\begin{array}{ccl}
i&&\hbox{if }i>0,\\
2d+1-i&&\hbox{if }i<0.
\end{array}\right.$$
The length $\lambda(\mathbf x)$
of $\mathbf x=(x_1,\dots,x_d,x_{-d},\dots,x_{-1})\in\mathbb F^{2d}$ 
is the maximal length 
of a basis-vector involved with non-zero coefficient. The length 
of an element depends crucially on the dimension $2d$ of the surrounding
symplectic space as shown by the following example: 
The element $3b_2+b_3-5b_{-3}$ of $(\mathbb F^6,\omega)$ with coordinate-vector 
$(0,3,1,-5,0,0)$ has length $4$.
The same element  $3b_2+b_3-5b_{-3}$ has length $6$ 
when considered as an element of $(\mathbb F^8,\omega)$
since its coordinate-vector is then given by $(0,3,1,0,0-5,0,0)$. 

Let $F=(V_1\subset V_2\subset \dots\subset
V_k)$ be a symplectic flag of $\mathbb F^{2d}$.
We associate to $F$ a basis $f_1,\dots,f_d$ of a 
canonically defined maximal Lagrangian 
$L$ containing $V_k$ and a length-permutation $\sigma\in 
\mathcal S_d^\pm$ as follows:
$f_1$ is the element of minimal length in $V_1\setminus\{0\}$
with last non-zero coordinate $1$.
We define $\sigma(1)$ as the unique element of $\{\pm 1,\dots,\pm d\}$ such 
that $f_1$ and $b_{\sigma(1)}$ have the same length.
More generally, given $f_1,\dots, f_i$ (with $i<d$) we define 
$f_{i+1}$ as follows: 
If $\dim(V_k)>i$, let $j\leq k$ be the smallest index such that
$\dim(V_j)>i$ and let $g_{i+1}$ be the shortest non-zero element
all whose non-zero coordinates have indices in 
$\{\pm 1,\dots,\pm d\}\setminus\{\pm\sigma(1),\dots,\pm \sigma(i)\}$,
which has a last non-zero coordinate with coefficient $1$ 
and which is such that there exist constants
$\lambda_1,\dots,\lambda_i$ in $\mathbb F$ with 
$g_{i+1}+\sum_{k=1}^i\lambda_kb_{-\sigma(k)}\in V_j\setminus\{\oplus_{k=1}^i\mathbb Ff_k\}$.
Set $f_{i+1}=g_{i+1}+\sum_{k=1}^i\lambda_kb_{-\sigma(k)}$
and define $\sigma(i+1)$ as the unique integer of 
$\{\pm 1,\dots,\pm d\}\setminus\{\pm \sigma(1),\dots,\pm \sigma(i)\}$
such that $g_{i+1}$ and $b_{\sigma(i+1)}$ have the same length.
If $\dim(V_k)\leq i<d$ proceed similarly by considering $g_{i+1}$
in $\mathbb F^{2d}\setminus \left(\oplus_{j=1}^i \mathbb F f_i\right)$
and define the constants $\lambda_1,\dots,\lambda_i$ in order
to get orthogonality (with respect to $\omega$) 
of $f_{i+1}=g_{i+1}+\sum_{k=1}^i\lambda_kb_{-\sigma(k)}$ with $f_1,\dots,f_i$.

In particular $\sigma(d)<0$ implies $\dim(V_k)=d$ 
(and the flag $F$ is maximal).

An important point of this construction is the observation that 
the generally non-zero coordinates 
$x_{-\sigma(1)},x_{-\sigma(2)},\dots,x_{-\sigma(i)}$ of 
$f_{i+1}$ are not involved in the ``length'' of $g_{i+1}$ 
defining $\sigma(i+1)$. These coordinates are only used for forcing 
the orthogonality relations
$\omega(f_{i+1},f_1)=\dots=\omega(f_{i+1},f_i)=0$. 

Observe that the length-permutation $\sigma$ belongs to the 
subgroup $\mathcal S_d$ of ordinary permutations in $\mathcal S_d^\pm$
if and only if the flag $F$ is contained in the maximal Lagrangian
generated by $b_1,\dots,b_d$. The construction of the
canonical basis $f_1,\dots,f_d$ coincides then with the construction of 
the canonical basis of $F$ considered as an ordinary (non-symplectic)
flag of $\mathbb F^d=\mathbb F b_1+\dots+\mathbb F b_d$.

Observe also that the lengths 
of $b_{\sigma(i)}$ and $b_{\sigma(i+1)}$
are increasing if $f_i,f_{i+1}\in V_j\setminus V_{j-1}$ for some $j$.
Two such consecutive integers contribute however a descent to the flag-Major index
of $\sigma$ if $\sigma(i)>0>\sigma(i+1)$.

We call a maximal flag $F=(V_1,\dots,V_k)$ of $\mathbb F^{2d}$ \emph{standard}
if $\sigma(d)<0$ and 
if $b_{\sigma(\dim(V_i))}$ is longer than $b_{\sigma(\dim(V_i)+1)}$ whenever $i<k$.
A non-maximal flag $F=(V_1,\dots,V_k)$ is \emph{standard}
if $b_{\sigma(\dim(V_i))}$ is longer than $b_{\dim(V_i)+1}$ for $i\leq k$
(recall that a flag is maximal if and only if $V_k$ is $d$-dimensional).

Every symplectic flag $V_1\subset \dots \subset V_k$
contains a unique standard subflag $V_{i_1}\subset\dots\subset V_{i_j}$ (for
some integer $j\leq k$ and for $i_1<i_2<\dots<i_j$ a subset of $\{1,\dots,j\}$)
with the same length-permutation.

The \emph{ standard weight} of a standard flag $F=(V_1\subset \dots\subset V_k)$ is given in the usual way by
$\sum_{i=1}^k \dim(V_i)$ and is the minimal  weight 
for $F$ considered as a weighted flag.

Theorem \ref{thmsympl} will be an easy consequence of the following 
generalization of Theorem \ref{corstflags}:

\begin{thm}\label{thmstandsympl}
We have
$$\sum_{F\in\mathcal{F}_{st}(\mathbb F_q^{2d},\omega)}t^{w_{st}(F)}=
\sum_{\sigma\in\mathcal S_d^{\pm}}q^{l^\pm(\sigma)}t^{\mathop{Wmaj}(\sigma)}$$
where $\mathcal{F}_{st}(\mathbb F_q^{2d},\omega)$
denotes the set of all standard flags of $(\mathbb F_q^{2d},\omega)$ and where
$w_{st}(F)$ is the standard weight of a standard flag $F$.
\end{thm}

Theorem \ref{thmstandsympl} will be proven in the next section by
showing that the set of all 
standard flags of $(\mathbb F_q^{2d},\omega)$ with length-permutation $\sigma
\in \mathcal S_d^\pm$ contains 
$q^{l^\pm(\sigma)}$ elements and that all these standard flags have the
same standard weight $\mathop{Wmaj}(\sigma)$.

\subsection{Half-bases with given length-permutation}

We denote by $\mathcal{B}(\sigma)$ the set of all half-bases
of $\mathbb F^{2d}$ with associated length-permutation $\sigma$.
The set $\mathcal{B}(\mathop{id})$ for example corresponds to 
the half-basis $b_1,b_2,\dots,b_d$ associated to the empty flag.
The half-basis $b_{-1},b_{-2},\dots,b_{-d}$ belongs to $\mathcal B(c)$
where $c(i)=-i$ is the Coxeter element. The standard 
flag associated to $b_{-1},b_{-2},\dots,b_{-d}$
is the complete flag $V_1,\dots,V_i=\oplus_{j=1}^i \mathbb Fb_{-j},\dots,
V_d=\oplus_{j=1}^d \mathbb Fb_{-j}$. The set
$\mathcal B(c)$ corresponds to complete generic flags and is in some
sense as large as possible.

\begin{prop}\label{propdimhalfbase} The set $\mathcal{B}(\sigma)$ of all half-bases
of $\mathbb F^{2d}$ with associated length-permutation $\sigma$
is in one-to-one correspondence with the vector space $\mathbb F^{l(\sigma)}$
where $l(\sigma)$ is given by Formula (\ref{defsignedlength}) of 
Proposition \ref{propsignedlength}.
\end{prop}

\begin{proof} (We suggest to contemplate the example given
in Section \ref{sectRothedsympl} while reading the proof.)
Elements $f_1,\dots,f_d$ of $\mathcal B(\sigma)$
are of the following form:
$f_i$ has a coefficient $1$ (represented by a bullet $\bullet$ in 
the example of Section \ref{sectRothedsympl}) at position $\sigma(i)$, has
coefficients $0$ (represented by empty spaces in 
the example of Section \ref{sectRothedsympl})
at positions $\sigma(1),\dots,\sigma(i-1)$ 
and at positions not of the form $-\sigma(1),\dots,-\sigma(i-1)$
which follow $\sigma(1)$ (with indices in the order 
$1,2,3,\dots,d,-d,-d+1,-d+2,\dots,-2,-1$ given by $<_\pm$) and
has arbitrary coefficients (represented by $\times$ or $\otimes_*$ in 
Section \ref{sectRothedsympl}, see below for explanations)
at positions not of the form
$\pm \sigma(1),\dots,\pm\sigma(i-1)$ which precede $\sigma(i)$.
The coefficients at positions $-\sigma(1),\dots,-\sigma(i-1)$ 
(represented by $\perp$ in 
Section \ref{sectRothedsympl}) are uniquely
determined by orthogonality of $f_i$ with $f_1,\dots,f_{i-1}$.

Proposition \ref{propdimhalfbase} will follow from the fact that
the number of arbitrary free coefficients of such a half-basis
$f_1,\dots,f_d$ in $\mathcal B(\sigma)$ equals the length 
$l(\sigma)$ of $\sigma$.

Free coefficients of $f_i$ are of two types: 
We call such a coefficient involved in $f_i$ of \emph{inversion-type}
(represented by a cross $\times$ in 
the example of Section \ref{sectRothedsympl})
if it corresponds to an index $k$ before $\sigma(i)$
with $j=\sigma^{-1}(k)>i$ (ie. if there exists
$j>i$ such that the index $\sigma(j)=k$ comes before
$\sigma(i)$ with respect to the order 
$1,2,3,\dots,d,-d,-d+1,\dots,-2,-1$ given by $<_\pm$).

The total number of such coefficients (which are all free)
is given by the inversion contribution
$\sum_{0<i<j,\sigma(i)>\pm \sigma(j)}1$
to $l(\sigma)$ in Formula (\ref{defsignedlength}). 

The remaining free coefficients (represented by 
tensor products $\otimes$ in Section \ref{sectRothedsympl}) 
of $f_i$ correspond to indices $k<_\pm 
\sigma(i)$ such that $\sigma(-k)\geq i$. We call such indices of 
\emph{sign-type}. If $\sigma(i)>0$, free coefficients of sign-type
involved in $f_i$ correspond to integers $j>i$ such that 
$0<-\sigma(j)<\sigma(i)$. If $\sigma(i)<0$, free
coefficients of sign-type have indices $-\sigma(j)$ for $j>i$ 
such that $\sigma(j)\vert>-\sigma(i)>0$ and $-\sigma(j)$ for $j\geq i$
such that $\sigma(j)<0$. (In particular, if $\sigma(i)<0$, the 
coefficient of $-\sigma(i)$ is always free of sign-type).
The following simplified Rothe diagrams (with the second double vertical bar separating 
positive from negative indices) resume the different cases 
for free coefficients of sign-type (the meaning of the indices 
will be explained later):
$$\begin{array}{c}
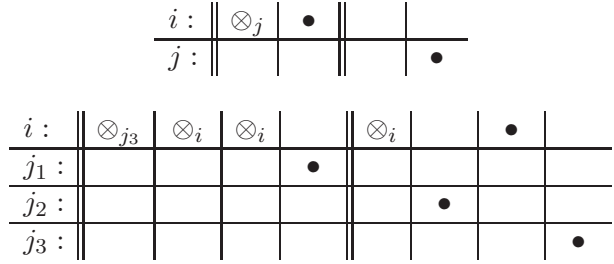
\begin{array}{l||c|c||c|c}
i:&\otimes_j&\bullet&\phantom{\otimes_j}&\phantom{\otimes_j}\\
\hline
j:&&\phantom{\otimes_j}&&\bullet
\end{array}
\\
\\
\begin{array}{l||c|c|c|c||c|c|c|c}
i:&\otimes_{j_3}&\otimes_i&\otimes_{i}&\phantom{\otimes_{i_i}}&
\otimes_{i}&\phantom{\otimes_{i_i}}&
\bullet&\phantom{\otimes_{i_i}}\\
\hline
j_1:&&\phantom{\otimes_{i_i}}&&\bullet&&&\phantom{\otimes_{i_i}}&\\
\hline
j_2:&&&&&&\bullet&&\\
\hline
j_3:&&&&&&&&\bullet\\
\end{array}

\end{array}$$

We will show that the number of coefficients of sign-type 
is given by the last summand
$\sum_{0<i,\sigma(i)<0}(d+1+\sigma(i))$ of Formula 
(\ref{defsignedlength}) defining $l(\sigma)$. 

We associate an integer (written as an index of $\otimes$ 
in Section \ref{sectRothedsympl}) in the set $\{1\leq i\leq d\ \sigma(i)<0\}$ 
to each free coefficient of sign-type in the following way:
A free coefficient of sign-type is associated to $i$ (with $\sigma(i)<0$) if 
it corresponds either to a free coefficient of $f_i$ with $\sigma(i)<0$ whose
index belongs to the set 
$$\{-\sigma(i),\pm (1-\sigma(i)),\dots,\pm d\}$$
or if it corresponds to a coefficient indexed by $-\sigma(i)$ of $f_j$ 
with $j<i$ such that $\vert\sigma(j)\vert>-\sigma(i)$.

We claim that every free coefficient of sign-type is associated to an integer 
$i$ with $\sigma(i)<0$ and that there are exactly $d+1+\sigma(i)$
such free coefficients of sign-type associated to $i$. 
This ends the proof since it implies the existence of exactly
$\sum_{i>0,\sigma(i)<0}d+1+\sigma(i)$ free coefficients of sign-type.

Let us first consider a free coefficient of sign-type, 
say a coefficient of index $k$ involved in $f_i$. 
Since it is a free coefficient, we have $k<_\pm \sigma(i)$. 
Suppose first that $\sigma(i)>0$. This implies $1<k<\sigma(i)$.
Since it is a free coefficient of sign-type, there exists $j>i$ such that
$k=-\sigma(j)<\sigma(i)$ and the coefficient is associated to $j$.
Suppose now $\sigma(i)<0$. There exists again $j$ such that $k=-\sigma(j)$.
If $k<0$, then $\vert k\vert=\vert\sigma(j)\vert >\vert \sigma(i)\vert$ 
and the coefficient is associated to $i$.
If $k>0$ then the free coefficient is associated to $i$ if $k\geq -\sigma(i)$
and to $j$ otherwise.

At last we have to show that there are
$d+1+\sigma(i)$ free coefficients of sign-type associated to an 
integer $i$ such that $\sigma(i)<0$. This is realized by
a bijection between such coefficients and the set $\{-\sigma(i),-\sigma(i)+1,
\dots,d-1,d\}$ of all $d+1+\sigma(i)$ integers between $-\sigma(i)$ and $d$.
Indeed, let $k$ be such an integer. There exists $j$ such that 
$\vert\sigma(j)\vert=k$. If $j>i$ then the coefficient of $b_{-k}$ in $f_i$ 
is of sign-type. If $j\leq i$ then the coefficient of $b_{-\sigma(i)}$ in
$f_j$ is of sign-type.
\end{proof}

\subsection{Rothe diagrams}\label{sectRothedsympl}

The proof of Proposition \ref{propdimhalfbase}
can be visualized by generalizing the notion of a Rothe
diagram to the symplectic setting. The 
Rothe diagram of the signed permutation 
$(\sigma(1),\dots,\sigma(6))=(-5,3,-1,6,4,-2)$
is then given by 
$$\begin{array}{c||c|c|c|c|c|c||c|c|c|c|c|c|}
i\backslash \sigma(i)&1&2&3&4&5&6&-6&-5&-4&-3&-2&-1\\
\hline\hline
1&\otimes_3&\otimes_6&\times&\times&\otimes_1&\times&\otimes_1&\bullet&&&&\\
\hline
2&\otimes_3&\otimes_6&\bullet&&\perp&&&&&&&\\
\hline
3&\otimes_3&\otimes_3&&\times&\perp&\times&\otimes_3&&\otimes_3&\perp&\times&\bullet\\
\hline
4&\perp&\otimes_6&&\times&\perp&\bullet&&&&\perp&&\\
\hline
5&\perp&\otimes_6&&\bullet&\perp&&\perp&&&\perp&&\\
\hline
6&\perp&\otimes_6&&&\perp&&\perp&&\perp&\perp&\bullet&\\
\hline
\end{array}$$
Crosses $\times$ (corresponding to 
the seven inversions 
$$(1,3),(1,4),(1,5),(3,4),(3,5),(3,6),(4,5)$$
defining intersections of rows and columns delimited with bullets)
and tensor products $\otimes_i$ (with indices $i=1,3,6$ corresponding to the $3$ negative 
values $\sigma(1)=-5,\sigma(3)=-1,\sigma(6)=-1$) 
are arbitrary elements, bullets $\bullet$ represent $1$,
symbols for perpendicularity $\perp$ are coefficients determined uniquely 
by orthogonality relations and empty cases 
represent coefficients which are $0$.
The length $l(\sigma)$ of $\sigma$ is given by Formula 
(\ref{defsignedlength}) and equals the number of crosses $\times$ 
(corresponding to the seven signed inversions 
$$1<2,1<4,1<5,3<4,3<5,3<6,4<5$$
of $\sigma$ whose number is counted by   
the first summand of (\ref{defsignedlength})) added to 
the number of tensor products $\otimes_i$ for $i\in \{1,3,6\}$ (corresponding
to the three negative values $\sigma(1)=-5,\sigma(3)=-1,\sigma(6)=-2$). 
More precisely, the number of symbols $\otimes_i$ is given by 
$6+1+\sigma(i)$ for $i\in\{1,3,6\}$.

\section{Standard weights of standard flags associated
to elements of $\mathcal B(\sigma)$}

\begin{prop}\label{propsignedstandardweight} For every element $\sigma$ of $\mathcal S_d^\pm$,
the  standard weight of a standard flag associated
to a half-basis in $\mathcal B(\sigma)$ is given by the Weyl-Major index 
$$\mathop{Wmaj}(\sigma)=\sum_{i>0,\sigma(i+1)<\sigma(i)}i+\sum_{i>0,\sigma(i)<0}1$$
(cf. Formula (\ref{defsignedmaj})) of $\sigma$.
\end{prop}

We illustrate Proposition \ref{propsignedstandardweight} by 
the example of Section \ref{sectRothedsympl}. Standard flags 
associated to $(\sigma(1),\dots,\sigma(6))=(-5,3,-1,6,4,-2)$ are
given by $(V_1,V_2,V_3,V_4)$ where
\begin{align*}
V_1&=\mathbb F f_1,\\
V_2&=\mathbb F f_1+\mathbb F f_2+\mathbb F f_3,\\
V_3&=\mathbb F f_1+\mathbb F f_2+\mathbb F f_3+\mathbb F f_4,\\
V_4&=\mathbb F f_1+\mathbb F f_2+\mathbb F f_3+\mathbb F f_4+\mathbb F f_5+\mathbb F f_6
\end{align*}
with standard weight $\sum_{i=1}^4\dim(V_i)=1+3+4+6=14$.

The ascending runs of $\sigma$ are
$$\begin{array}{c|c}
\hbox{run}&\sigma(\hbox{run})\\
\hline
1,2&-5,3\\
3,4&-1,6\\
5&4\\
6&-1
\end{array}$$
and correspond to the composition
$2+2+1+1$ of $d=6$. The Weyl-Major index of $\sigma$ equals thus 
also
$$(2+4+5)+3=14$$
(with the summand $3$ corresponding to the three elements $i$ such 
that $\sigma(i)<0$).

\begin{proof}[Proof of Proposition \ref{propsignedstandardweight}]
We consider a symplectic standard flag $F=(V_1\subset\dots\subset V_{k-1}\subset V_k)$ with associated length-permutation $\sigma\in \mathcal S_d^\pm$.
We have to show that $F$ has standard weight $\mathop{Wmaj}(\sigma)$.

We prove Proposition \ref{propsignedstandardweight} by induction 
on $d$. If $d=1$, the unique element of $\mathcal B(\mathrm{id})$
corresponds to the empty standard flag of  weight $\mathop{Wmaj}(\mathrm{id})=0$ and elements of $\mathcal B(c)$ associated to the Coxeter element 
$c(1)=-1$ define complete flags $\mathbb F(x,1)$ in $(\mathbb F^2,\omega)$. 
They are standard flags
of  standard weight $1=\mathop{Wmaj}(c)$.

Given an integer $d\geq 2$, we denote by $\hat\sigma\in \mathcal S_{d-1}^\pm$ the signed permutation 
obtained by ``erasing'' $\sigma(d)$. More precisely, $\hat\sigma(i)$ for 
$i\in\{1,\dots,d-1\}$ is defined 
by 
$$\\\hat\sigma(i)=\left\lbrace
\begin{array}{ll}
\sigma(i)&\hbox{if }\vert\sigma(i)\vert<\vert\sigma(d)\vert,\\
\sigma(i)-1\quad&\hbox{if }\sigma(i)>\vert\sigma(d)\vert,\\
\sigma(i)+1&\hbox{if }\sigma(i)<-\vert\sigma(d)\vert\ .
\end{array}\right.$$ 

We transform the canonical half-basis $f_1,\dots,f_d$
associated to $F$ into a canonical half-basis $\hat f_1,\dots,\hat f_{d-1}$
of $\mathbb F_q^{2(d-1)}$
with associated length-permutation $\hat\sigma$ as follows: 
We remove first the coefficients corresponding to indices $\pm \sigma(d)$
from all vectors $f_1,\dots,f_{d-1}$. 
The destroyed orthogonality (with respect to the symplectic form) 
is now restored by correcting 
the coefficients of $f_i$ with indices 
$-\sigma(1),-\sigma(2),\dots,-\sigma(i-1)$ 
(the correction is defined uniquely). Finally, we rename indices 
of absolute value larger than $\vert \sigma(d)\vert$ 
by keeping their signs and 
decreasing their absolute value by $1$. We call the resulting
elements $\hat f_1,\dots,\hat f_{d-1}$.
We denote the standard flag associated to $\hat f_1,\dots,\hat f_{d-1}$
by $\hat F$. It is obviously associated with the length-permutation 
$\hat \sigma$.

The induction step follows now from the equality
$$\mathop{Wmaj}(\sigma)-\mathop{Wmaj}(\hat\sigma)=w_{st}(F)-w_{st}(\hat F)$$
(with $w_{st}$ denoting the  standard weight of a standard flag) which we are going to establish.

The proof splits into six cases given by the relative positions of
the three integers $0,\sigma(d-1),\sigma(d)$.
The following table resumes the six cases:
$$\begin{array}{c|c|c|c|c}
&&\Delta\mathop{Wmaj}&\hat V_\omega&\Delta w_{st}\\
\hline\hline
0<\sigma(d-1)<\sigma(d)&\begin{array}{cc|}
\bullet&\\
\hline
&\bullet\\
\end{array}&0&\hat V_k&0\\
\hline
\sigma(d-1)<0<\sigma(d)&\begin{array}{c|c}
&\bullet\\
\hline
\bullet&\\
\end{array}&0&\hat V_k&0\\
\hline\hline
\sigma(d)<0<\sigma(d-1)&\begin{array}{c|c}
\bullet&\\
\hline
&\bullet\\
\end{array}&d&\hat V_{k-1}&d\\
\hline
\sigma(d)<\sigma(d-1)<0&\begin{array}{|cc}
&\bullet\\
\hline
\bullet&\\
\end{array}&d&\hat V_{k-1}&d\\
\hline\hline
0<\sigma(d)<\sigma(d-1)&\begin{array}{cc|}
&\bullet\\
\hline
\bullet&\\
\end{array}&d-1&\hat V_{k-1}&d-1\\
\hline\hline
\sigma(d-1)<\sigma(d)<0&\begin{array}{|cc}
\bullet&\\
\hline
&\bullet\\
\end{array}&1&\hat V_k&1\\
\hline
\end{array}$$
The first column describes all six possible relative positions of 
$0,\sigma(d-1),\sigma(d)$. The second column depicts them graphically
using the conventions for Rothe diagrams (positive indices are separated from
negative indices by a vertical bar). The third and 
fifth columns contain the differences $\mathop{Wmaj}(\sigma)-\mathop{Wmaj}(\hat\sigma)$, respectively $w_{st}(F)-w_{st}\hat F)$. The fourth column gives the index $\omega$ 
of the largest space in the standard flag $\hat F$.

The cases $0<\sigma(d-1)<\sigma(d)$ and $\sigma(d-1)<0<\sigma(d)$ are similar:
$V_k$ does not involve $f_k$ and corresponds to $\hat V_k$.
The standard flags $F$ and $\hat F$ have identical  weights (with respect
to the standard weight). An easy computation shows that $\sigma$ and 
$\hat \sigma$ have identical Weyl-Major indices.

The cases $\sigma(d)<0<\sigma(d-1)$ and $\sigma(d)<\sigma(d-1)<0$ 
are also similar: The standard flag $\hat F$ ends in both cases with 
the ``projection'' $\hat V_{k-1}$ spanned by $\hat f_1,\dots,\hat f_{\dim(V_{k-1})}$
of $V_{k-1}$ while $F$ ends with a Lagrangian $V_k$ 
of dimension $d$. This implies
that  weights of $F$ and $\hat F$ differ by $d$ and an easy computation
shows that this holds also for the Weyl-Major indices.

In the case $0<\sigma(d)<\sigma(d-1)$, the space $V_k$ is spanned 
by $f_1,\dots,f_{d-1}$ while the largest space of $\hat F$ is 
$\hat V_{k-1}$ spanned by $\hat f_1,\dots,\hat f_{\dim(V_{k-1})}$.
This implies a difference of $d-1$ for the  weights of the 
standard flags 
$F$ and $\hat F$. The same difference is realized by the flag-Major indices
$\mathop{Wmaj}(\sigma)$ and $\mathop{Wmaj}(\hat\sigma)$ of the corresponding
length-permutations.

Finally, in the last case $\sigma(d-1)<\sigma(d)<0$, the space 
$\hat V_{k}$ is a maximal Lagrangian of dimension $d-1$ while 
$V_k$ is a maximal Lagrangian of dimension $d$. This implies
a difference of $1$ between  weights and the same difference of $1$
is also realized by Weyl-Major indices.
\end{proof}

\subsection{Proof of Theorem \ref{thmsympl}}

\begin{proof} Analogues of Proposition \ref{propweightedflagstandardflag}
and Corollary \ref{corpartitionstandard} hold obviously in 
the symplectic setting.
\end{proof}

\section{Type B: Orthogonal groups of odd dimension}
\label{secttypeB}

For simplicity, we work again only over finite fields of odd characteristic.
%We endow $\mathbb F_q^{2d+1}$ with a non-degenerate quadratic 
%form $Q$ having maximal isotropic subspaces of dimension $d$. 

\begin{proof}[Sketch of Proof of Theorem \ref{thmBequalC}]
We consider again the order $<_\pm$, see (\ref{pmorderonZ}).
Coordinate-vectors of an element in $\mathbb F^{2d+1}$
are always given in the form
$$(x_1,x_2,\dots,x_d,x_0,x_{-d},x_{1-d},\dots,x_{-2},x_{-1})\ .$$
We endow the vector space
$\mathbb F^{2d+1}$ generated by
$b_{\pm 1},\dots,b_{\pm d},b_0$
over a field of odd characteristic with the quadratic
form 
$$Q(x_1,\dots,x_d,x_0,x_{-d},\dots,x_{-1}))=x_0^2+
\sum_{i=1}^d x_ix_{-i}\ .$$

We consider the following bijection between elements of the set 
$\mathcal B_C(\lambda)$ of canonical bases for the symplectic space
and the set $\mathcal B_B(\lambda)$ of canonical bases (defined in
the obvious way) for the quadratic space $(\mathbb F_q^{2d+1},Q)$.

The coordinates $(x_1,\dots,x_d,x_0,x_{-d},\dots,x_{-1})$ 
of the $i$-th vector $f_i$ in a basis 
in $\mathcal B_B(\lambda)$ are as follows:
Coordinates with indices $<_\pm \lambda(i)$ in $\{0,\pm 1,\dots,\pm d\}\setminus
\{\pm \lambda(1),\dots,\pm \lambda(i)\}$
are free. The coordinate $x_{\lambda(i)}$ equals $1$.
The coordinates $x_{-\lambda(1)},\dots,x_{-\lambda(i-1)}$ are determined by 
orthogonality of $f_i$ to $f_1,\dots,f_{i-1}$.
The coordinate $x_{-\lambda(i)}$ of $f_i$ is determined by isotropy of $f_i$
(it is in fact always equal to $0$ if $\lambda(i)>0$). All coordinates $>_\pm \lambda(i)$
with indices in $\{\lambda(1),\dots,\lambda(i-1)\}$ or with indices in
$\{0,\pm 1,\dots,\pm d\}\setminus\{\pm \lambda(1),\dots,\pm \lambda(i)\}$
are zero. Observe that $x_{-\lambda(i)}$ is never free. However
$x_0$ is free if $\lambda(i)<0$ and it is equal to $0$ otherwise.
This shows that elements of $\mathcal B_B(\lambda)$
behave exactly in the same way as elements of $\mathcal B_C(\lambda)$.
We leave the easy details to the reader.
\end{proof}

We illustrate this Section by the Rothe diagram ``of type B'' associated 
to the signed permutation 
$(\sigma(1),\dots,\sigma(6))=(-5,3,-1,6,4,-2)$. Using conventions analogous to those of Section \ref{secthalfbases} it is given by
$$\begin{array}{c||c|c|c|c|c|c||c||c|c|c|c|c|c|}
i\backslash \sigma(i)&1&2&3&4&5&6&0&-6&-5&-4&-3&-2&-1\\
\hline\hline
1&\otimes_3&\otimes_6&\times&\times&\perp&\times&\otimes_1&\otimes_1&\bullet&&&&\\
\hline
2&\otimes_3&\otimes_6&\bullet&&\perp&&&&&&\perp&&\\
\hline
3&\perp&\otimes_3&&\times&\perp&\times&\otimes_3&\otimes_3&&\otimes_3&\perp&\times&\bullet\\
\hline
4&\perp&\otimes_6&&\times&\perp&\bullet&&\perp&&&\perp&&\\
\hline
5&\perp&\otimes_6&&\bullet&\perp&&&\perp&&\perp&\perp&&\\
\hline
6&\perp&\perp&&&\perp&&\otimes_6&\perp&&\perp&\perp&\bullet&\\
\hline
\end{array}\ .$$

\section{Type D: Orthogonal groups in even dimensions}

As in Section \ref{secttypeB}, all finite fields are of odd characteristic.

We denote by $\mathcal H=\mathcal H(\mathbb F)$ the hyperbolic plane
over a field $\mathbb F$
realized as the quadratic space $\mathbb F^2$ endowed with the quadratic 
form $(x,y)\longmapsto xy$, called \emph{norm} in the sequel.

We denote by $\mathcal I$ a fixed maximal $d$-dimensional 
isotropic subspace of $\mathcal H^d$ (denoting $d$ orthogonal copies
of $\mathcal H$). As always, a 
flag $F=(V_1\subset \dots\subset V_k)$ is a strictly increasing sequence
of non-trivial isotropic subspaces of $\mathcal H^d$. The 
\emph{$\mathcal{I}$-parity} of a flag 
ending with $V_k$ is the parity of the integer $\dim(V_k/(V_k\cap \mathcal I))
=\dim(V_k)-\dim(V_k\cap \mathcal I)$. Flags of even parity are simply
called even flags. We denote by ${\mathcal F}^e(\mathcal H^d)$
the set of all even flags and by $\mathcal{WF}^e(\mathcal H^d)$
the set of all weighted even flags.

\begin{rem} The parity condition of type D flags has the following explanation:
Two complete flags of type D are related by an isometry of determinant $1$
if and only if they are both even or both odd. Otherwise an isometry of
determinant $-1$ is needed. The set of all complete flags of type D
decays thus into two orbits under the action of the 
simple linear group of Lie type D.

In contrast, complete flags of 
type A,B,C are always related by an isomorphism (of the corresponding 
structure) of determinant $1$ and the action of the 
corresponding simple linear group is thus transitive.
\end{rem}

\subsection{Proof of Corollary \ref{corformMD}}

\begin{prop}\label{propnbrflagsD} The space $\mathcal H^d$ over the finite field $\mathbb F_q$
contains exactly
$$q^{l(2d+l-2k-1)/2}{d\choose k}_q{k\choose l}_q$$
isotropic subspaces $V$ of dimension $k$ such that $\dim(V/(V\cap \mathcal I))=l$.
\end{prop}

\begin{proof} There are ${d\choose l}_q$ different $l$-dimensional subspaces
of the $d$-dimensional quotient space $\mathcal H^d/\mathcal I$.
Such a subspace $A\subset \mathcal H^d/\mathcal I$
can be lifted in $q^{dl-{l+1\choose 2}}$ ways into an $l$-dimensional
isotropic subspace $\tilde A$ of $\mathcal H^d$.
We add to $\tilde A$ a $k-l$ dimensional 
subspace $B$ of $\mathcal I\cap \tilde A^\perp$. This can be done
in ${d-l\choose k-l}_q$ different ways.
The subspace $V=\tilde A+B$ has the required properties.
Such subspaces arise however with multiplicity $q^{l(k-l)}$. 
Indeed, for $B,B'\subset \tilde A^\perp\cap \mathcal I$, the equality 
$\tilde A+B=\tilde A'+B'$ holds if and only if $B=B'$ and
$\tilde A'=\tilde A\pmod B$.
The total number of such spaces is thus given by 
$${d\choose l}_q{d-l\choose k-l}_qq^{dl-{l+1\choose 2}-l(k-l)}$$
which amounts to the formula given by Proposition \ref{propnbrflagsD}.
\end{proof}

\subsection{First values and some properties of the polynomials 
$M_d^D$ for type D}

The first few polynomials $M_d^D$ are as follows: 
$M^D_1=1$, coefficients of $M^D_2,M^D_3$ are given by
$$
{\begin{array}{c|ccc}
&1&q&q^2\\
\hline
1&1\\
t&&1&\\
t^2&&1&\\
t^3&&&1\\
%\hline
%&1&2&2&2&1
\end{array}}
\quad
{\begin{array}{c|ccccccc}
&1&q&q^2&q^3&q^4&q^5&q^6\\
\hline
1&1\\
t&&1&1\\
t^2&&1&1&1&1&1\\
t^3&&1&1&2&1&1&1\\
t^4&&&1&2&2&1\\
t^5&&&1&1&1\\
\end{array}}
$$
and coefficients of $M^\pm_4$ are given by
$${\begin{array}{c|ccccccccccccc}
&1&q&q^2&q^3&q^4&q^5&q^6&q^7&q^8&q^9&q^{10}&q^{11}&q^{12}\\
\hline
1&1\\
t&&1&1&1\\
t^2&&1&2&1&1&1&1&2&1&1\\
t^3&&1&1&3&3&4&4&3&3&1&1\\
t^4&&1&2&3&4&5&6&6&5&3&1\\
t^5&&&1&3&6&7&8&7&6&3&1\\
t^6&&&1&3&5&6&6&5&4&3&2&1\\
t^7&&&1&1&3&3&4&4&3&3&1&1\\
t^8&&&&1&1&2&1&1&1&1&2&1\\
t^9&&&&&&&&&&1&1&1\\
t^{10}&&&&&&&&&&&&&1\\
\end{array}}$$

Evaluation of $M_d^D$ at $t=1$ yields
$$\frac{1-q^d}{1-q}\prod_{j=1}^{d-1}\frac{1-q^{2j}}{1-q}\ .$$

(Proof amounts to equality
$$\sum_{l=0}^{\lfloor d/2\rfloor}{d\choose 2l}_qq^{l(2l-1)}=\prod_{j=1}^{d-1}
(1+q^j)\quad .)$$

Evaluating the obvious identity 
\begin{align*}
%\label{formulaflagMahtypeD}
M_d^\pm=\sum_{k=0}^dt^k\left(\prod_{j=k+1}^d(1-t^j)\right){d\choose k}_q
\left(\sum_{l=0}^{k}{k\choose l}_qq^{l(2d+l-2k-1)/2}\right)M_k
\end{align*}
at $t=1$ and comparing with ??? we get
$$\sum_{j=0}^d{d\choose j}_qq^{j\choose 2}=\prod_{j=0}^{d-1}(1+q^j)\ .$$

The following result generalizes the classical binomial theorem 
(corresponding to the case $q=1$):
\begin{thm}\label{lemqbinid}
We have 
$$\sum_{j=0}^d{d\choose j}_qq^{j+a\choose 2}t^j=q^{a\choose 2}\prod_{j=0}^{d-1}(1+tq^{j+a})\ .$$
for all $d\in\mathbb N$.
\end{thm}
\begin{proof}
The identity holds trivially for $d=0$. The recursive 
definition ${d+1\choose j}_q={d\choose j-1}_q+q^j{d\choose j}_q$ implies
\begin{align*}
&\quad \sum_{j=0}^{d+1}{d+1\choose j}_qq^{j+a\choose 2}t^j\\
&=\sum_{j=1}^{d+1}{d\choose j-1}_qq^{(j-1)+(a+1)\choose 2}t^{(j-1)+1}+
q^{-a}\sum_{j=0}^{d}{d\choose j}_qq^{j+a+{j+a\choose 2}}t^j\\
&=(t+q^{-a})q^{a+1\choose 2}\prod_{j=0}^{d-1}(1+q^{a+1+j})\\
&=q^{a\choose 2}\prod_{j=0}^d(1+tq^{a+j})
\end{align*}
which ends the proof by induction.
\end{proof}

\begin{proof}[Proof of Theorem \ref{formflagMajDqeqone}]
Evaluating $M_d^D$, given by Corollary \ref{corformMD}, at $q=1$ we get
\begin{align*}
&\quad \prod_{j=1}^d(1-t^j)+\sum_{k=1}^dt^k\prod_{j=k+1}^d(1-t^j){d\choose k}2^{k-1}\prod_{j=1}^k\frac{1-t^j}{1-t}\\
&=\prod_{j=1}^d(1-t^j)\left(1+\frac{1}{2}\sum_{k=1}^d\left(\frac{2t}{1-t}\right)^k{d\choose k}\right)\\
&=\prod_{j=1}^d(1-t^j)\frac{1}{2}\left(1+\left(1+\frac{2t}{1-t}\right)^d\right).
\end{align*}
Simplification yields
\begin{align*}
%\label{formulaflagstatD}
&\frac{(1-t)^d+(1+t)^d}{2}\prod_{j=1}^d\frac{1-t^j}{1-t}\ .
\end{align*}
which ends the proof.
\end{proof}

\section{The length function for type D}

We consider the usual generators of the Weyl group of type D given 
by $s_i=(i,i+1),\ i<d$ and $s_d(d-1)=-d,s_d(d)=1-d,s_d(i)=i$ for $i\not\in
\{\pm (d-1),\pm d\}$.

\begin{prop}\label{proplengthD} The length $l^D(\sigma)$ of an element $\sigma\in 
\mathcal S_d^D$ with respect to the generators
$s_1,\dots,s_d$ is given by the formula
\begin{align}\label{deflengthD}
l^D(\sigma)&=\sum_{0<i<j,\sigma(i)>\pm \sigma(j)}1+
\sum_{0<i,\sigma(i)<0}(d+\sigma(i))
\end{align}
with $>_\pm$ denoting the order-relation of $\mathbb Z$ defined by 
(\ref{pmorderonZ}).
\end{prop}
Observe that the length $l^D(\sigma)$ of an element $\sigma\in 
\mathcal S_d^D$ with respect to the generators $s_1,\dots,s_d\in S_d^D$
is always bounded above by its length $l(\sigma)$ with respect to the 
natural generators $s_1,\dots,s_d$ of $S_d^\pm$. Equality holds if and only
if $\sigma\in S_d$. More precisely, the difference is exactly 
the even number of elements in $-\mathbb N\cap\{\sigma(1),\dots,\sigma(d)\}$,
as can be seen by comparing Formula (\ref{deflengthD})
with Formula (\ref{defsignedlength}).

\begin{proof}[Proof of Proposition \ref{proplengthD}]
The proof is by induction on the length and 
completely analogous to the proof of Proposition \ref{propsignedlength}.

The result holds of course if $\sigma$ is the identity.

The crucial point for induction 
is again the equality $\vert l^D(\sigma)-l^D(\sigma\circ s_d)
\vert=1$. (The behaviour with respect to the $d-1$ first generators
$\sigma_1,\dots,\sigma_{d-1}$ is as in the proof of Proposition 
\ref{propsignedlength}.) 
We denote $l^D$ simply by $l$ until the end of the proof.

We write $a=\sigma(d-1)$ and $b=\sigma(d)$.

We consider two cases, depending on the sign of $ab$.

We discuss first the case $ab>0$. If $a>b$, we replace $\sigma$ by 
$\tilde\sigma=\sigma\circ s_{d-1}$. We have then $l(\sigma)=
l(\tilde \sigma)+1$ and $l(\sigma\circ s_d)=l(\tilde \sigma\circ s_d)+1$.
We can thus assume that $a<b$. Up to replacing $\sigma$ with $\sigma\circ s_d$,
we can furthermore assume that $0<a<b$. The following representation,
similar to the representation used in the proof of Proposition 
\ref{propsignedlength}, depicts all possible subcases:
$$\begin{array}{r||cccccc|cc}
\sigma(j)&i_1&i_2&i_3&ç_4&i_5&i_6&d-1&d\\
\hline\hline
&&&&&&\bullet_6&&\\
b&&&&&&&&\bullet\\
&&&&&\bullet_5&&&\\
a&&&&&&&\bullet&\\
&&&&\bullet_4&&&&\\
\hline
&&&\bullet_3&&&&&\\
-a&&&&&&&&\circ\\
&&\bullet_2&&&&&&\\
-b&&&&&&&\circ&\\
&\bullet_1&&&&&&&\\
\end{array}$$
(the horizontal line represents $0$, the vertical line 
separates the two last indices $d-1$ and $d$ from previous ones,
the last two values of $\sigma$, respectively of $\tilde \sigma=\sigma\circ s_d$, are represented by $\bullet$, respectively $\circ$).
We denote by $i_j$ indices taking values depicted by $\bullet_j$, i.e.
$$\sigma(i_1)<-b<\sigma(i_2)<-a<\sigma(i_3)<0<\sigma(i_4)<a<
\sigma(i_5)<b<\sigma(i_6)\ .$$
The following Table represents the status with respect to inversions
(Yes for inversions, No for non-inversions)
of $\sigma$ and $\tilde \sigma=\sigma\circ s_d$
for $i_j<(d-1)$ and for $i_j<d$:
$$\begin{array}{c||c|c||c|c}
j&\sigma(i_j)>_\pm \sigma(d-1)&\sigma(i_j)>_\pm \sigma(d)&
\tilde\sigma(i_j)>_\pm \tilde\sigma(d-1)&\tilde\sigma(i_j)>_\pm \tilde\sigma(d)\\
\hline
1&\mathrm{Yes}&\mathrm{Yes}&\mathrm{No}&\mathrm{No}\\
2&\mathrm{Yes}&\mathrm{Yes}&\mathrm{Yes}&\mathrm{No}\\
3&\mathrm{Yes}&\mathrm{Yes}&\mathrm{Yes}&\mathrm{Yes}\\
4&\mathrm{No}&\mathrm{No}&\mathrm{No}&\mathrm{No}\\
5&\mathrm{Yes}&\mathrm{No}&\mathrm{No}&\mathrm{No}\\
6&\mathrm{Yes}&\mathrm{Yes}&\mathrm{No}&\mathrm{No}\\
\end{array}$$
Setting
$$\nu_j=\sharp\{i<d-1\vert i\text{ is of type }i_j\}$$
we get now
$$\begin{array}{rcl}
l(\sigma\circ s_d)-l(\sigma)&=&d-a+d-b-2\nu_1-\nu_2-\nu_5-2\nu_6\\
&=&2d-a-b-(2(\nu_1+\nu_6)+(\nu_2+\nu_5))\\
&=&2d-a-b-(2(d-b)+(b-1-a))\\
&=&1
\end{array}$$
where we have used the trivial identities $\nu_1+\nu_6=d-b$
and $\nu_2+\nu_5=b-1-a$. This settles the case $ab>0$.

We consider now the case of $ab<0$ with $a=\sigma(d-1), b=\sigma(d)$.
If $a<0<b$, we set $\tilde \sigma=\sigma\circ s_{d-1}$.
Since $l(\sigma)-l(\tilde \sigma)=1$ and 
$l(\sigma\circ s_d)-l(\tilde \sigma\circ s_d)=1$ we can replace
$\sigma$ with $\tilde \sigma$ without loss of generality.
We can thus assume $a=\sigma(d-1)>0>-b=\sigma(d)$.
Up to replacing $\sigma$ with $\sigma\circ s_d$, we can moreover 
assume that $a<b$. The situation is now represented by
$$\begin{array}{r||cccccc|cc}
\sigma(j)&i_1&i_2&i_3&ç_4&i_5&i_6&d-1&d\\
\hline\hline
&&&&&&\bullet_6&&\\
b&&&&&&&\circ&\\
&&&&&\bullet_5&&&\\
a&&&&&&&\bullet&\\
&&&&\bullet_4&&&&\\
\hline
&&&\bullet_3&&&&&\\
-a&&&&&&&&\circ\\
&&\bullet_2&&&&&&\\
-b&&&&&&&&\bullet\\
&\bullet_1&&&&&&&\\
\end{array}$$
The table describing inversions involving $d-1$ or $d$ is 
$$\begin{array}{c||c|c||c|c}
j&\sigma(i_j)>_\pm \sigma(d-1)&\sigma(i_j)>_\pm \sigma(d)&
\tilde\sigma(i_j)>_\pm \tilde\sigma(d-1)&\tilde\sigma(i_j)>_\pm \tilde\sigma(d)\\
\hline
1&\mathrm{Yes}&\mathrm{No}&\mathrm{Yes}&\mathrm{No}\\
2&\mathrm{Yes}&\mathrm{Yes}&\mathrm{Yes}&\mathrm{No}\\
3&\mathrm{Yes}&\mathrm{Yes}&\mathrm{Yes}&\mathrm{Yes}\\
4&\mathrm{No}&\mathrm{No}&\mathrm{No}&\mathrm{No}\\
5&\mathrm{Yes}&\mathrm{No}&\mathrm{No}&\mathrm{No}\\
6&\mathrm{Yes}&\mathrm{No}&\mathrm{Yes}&\mathrm{No}\\
\end{array}$$
(with $\tilde\sigma=\sigma\circ s_d$, as before).

Defining the numbers $\nu_i$ as above, we get
$$\begin{array}{rcl}
l(\sigma\circ s_d)-l(\sigma)&=&(d-a)-(d-b)-(\nu_2+\nu_5)\\
&=&b-a-(b-1-a)\\
&=&1
\end{array}$$
which settles the case $ab<0$.

Since $l(\sigma)$ and $l(\sigma\circ s_i)$ differ always exactly
by $1$, the length of an a element $\sigma$ in $\mathcal S_d^D$
is at least $l(\sigma)$.
Let now $\sigma$ be a non-trivial element of $S_d^D$ (we have
of course $l(id)=0$ for the identity
permuation $id$ of $S_d^D$). If $\sigma$ has an inversion
then it has an inversion involving two consecutive indices
$i,j=i+1$ and replacing $\sigma$ with $\sigma\circ s_i$
decreases its length $l$ by one. If $\sigma\not=id$ is without inversions
then it ends with $\sigma(d-1)=-2,\sigma(d)=-1$ and applying $s_d$
decreases its length by $1$.
\end{proof}

The proof of Proposition \ref{proplengthD} is again algorithmic. We illustrate
it by considering the permutation 
$(\sigma(1),\dots,\sigma(4))=(-2,4,-3,1)$ of $S_4^D$. It has length $8$
(the pairs $(1,2),(1,3),(1,4),(2,4),(3,4)$ define inversions and 
we get a two sign-contributions $4+\sigma(1)=4-2=2$ and $4+\sigma(3)=4-3=1$).
We denote a permutation $\tau$ of $S_4^D$ always by $(\tau(1),\dots,\tau(4))$.
We have 
$$\begin{array}{r|c|c}
\sigma&(-2,4,-3,1)&8\\
\sigma\circ s_3&(-2,4,1,-3)&7\\
\sigma\circ s_3\circ s_2&(-2,1,4,-3)&6\\
\sigma\circ s_3\circ s_2\circ s_1&(1,-2,4,-3)&5\\
\sigma\circ s_3\circ s_2\circ s_1\circ s_2&(1,4,-2,-3)&4\\
\sigma\circ s_3\circ s_2\circ s_1\circ s_2\circ s_4&(1,4,3,2)&3\\
\sigma\circ s_3\circ s_2\circ s_1\circ s_2\circ s_4\circ s_3&(1,4,2,3)&2\\
\sigma\circ s_3\circ s_2\circ s_1\circ s_2\circ s_4\circ s_3\circ s_2&(1,2,4,3)&1\\
\sigma\circ s_3\circ s_2\circ s_1\circ s_2\circ s_4\circ s_3\circ s_2\circ s_3&(1,2,3,4)&0\\
\end{array}$$
yielding the minimal expression 
$$\sigma=s_3\circ s_2\circ s_3\circ s_4\circ s_2\circ s_1\circ s_2\circ s_3$$
of $\sigma$ in terms of the generators 
$$s_1=(2,1,3,4),s_2=(1,3,2,4),s_3=(1,2,4,3),s_4=(1,2,-4,-3)\ .$$

\section{Halfbases for type D}

Half-bases for type D are similar to half-bases in the symplectic case.
The only difference is the fact that the coefficient of index $-\lambda(i)$
in $f_i$ is always determined by isotropy. (It is free in the symplectic 
case if $\lambda(i)<0$. This difference in behaviour 
translates to a difference 
of $1$ in the summands of the second summation occuring in 
Formulae (\ref{defsignedlength}) and (\ref{deflengthD}).) 
We illustrate this by the Rothe diagram 
of $(\sigma(1),\dots,\sigma(6))=(-5,3,-1,-6,4,-2)$
which is given by 
$$\begin{array}{c||c|c|c|c|c|c||c|c|c|c|c|c|}
i\backslash \sigma(i)&1&2&3&4&5&6&-6&-5&-4&-3&-2&-1\\
\hline\hline
1&\otimes_3&\otimes_6&\times&\times&\perp&\times&\otimes_1&\bullet&&&&\\
\hline
2&\otimes_3&\otimes_6&\bullet&&\perp&&&&&\perp&&\\
\hline
3&\perp&\otimes_3&&\times&\perp&\times&\otimes_3&&\otimes_3&\perp&\times&\bullet\\
\hline
4&\perp&\otimes_6&&\times&\perp&\perp&\bullet&&&\perp&&\\
\hline
5&\perp&\otimes_6&&\bullet&\perp&\perp&&&\perp&\perp&&\\
\hline
6&\perp&\perp&&&\perp&\perp&&&\perp&\perp&\bullet&\\
\hline
\end{array}$$

\subsection{Proof of Theorem \ref{mainthmtypeD}}\label{sectproofD}

Up to obvious modifications, the proof is as for the type $C$.

\section{Incorporating statistics for Eulerian polynomials of type A and BC}\label{sectEulerpol}

%\section{Taking into account the weight-sum $\sum_{i=1}^kw_i$}\label{sectEulerpol}

According to \cite{Br}, a 
\emph{descent} of an element $w$ in a Weyl group $W$ is a canonical 
generator $s_i$
such that $w\circ s_i$ is shorter than $w$. Using our conventions,
it is easy to check that the number of descents of $\sigma$ in 
$\mathcal S_d$ or in $\mathcal S_d^\pm$ 
is given by
\begin{align}\label{descnumber}
\beta(\sigma)&=\delta(\sigma(d)<0)+\sum_{1\leq i<d,\sigma(i)>_\pm\ \sigma(i+1)}1
\end{align}
where $\delta(\hbox{true})=1$ and $\delta(\hbox{false})=0$.

Formula (\ref{descnumber}) does not coincide with the number of 
descents in Weyl groups of type D (the difference is however always 
bounded by $1$).

We consider now the extended Weyl-Mahonian statistics defined by
$$\tilde M_d^*=\sum_{\sigma\in S_d^*}q^{l^*(\sigma)}s^{\beta(\sigma)}t^{Wmaj(\sigma)}$$
of type A,BC and D.

Given a weighted flag $F=(V_1\subset \dots\subset V_k;w_1,\dots,w_k)$, 
we set $\alpha(F)=\sum_{i=1}^k w_i$. We have obviously
$\alpha(F)\leq w(F)=\sum_{i=1}^k w_i\dim(V_i)$.

Straightforward modifications of the proofs of Theorems \ref{mainthm},
\ref{thmsympl} and \ref{mainthmtypeD} show easily the 
following result:

\begin{thm} We have
$$\sum_{F\in \mathcal{WF}(*,q)_d}s^{\alpha(F)}t^{w(F)}=
\tilde M_d^*\prod_{j=1}^d\frac{1}{1-st^j}$$
with $\mathcal{WF}(*,q)_d$ denoting the obvious set of weighted
flags of type $*$ with $*$ standing for A,BC or D.
\end{thm}

Formulae for $\tilde M_d^*$ are given by the following result:

\begin{thm}\label{thmEuler}
For type A we get
\begin{align}\label{formrecMdtilde}
\tilde M_d&=\left(\prod_{j=1}^{d-1}1-st^j\right)+s\sum_{k=1}^{d-1}t^k{d\choose k}_q\left(\prod_{j=k+1}^{d-1}1-st^j\right)
\tilde M_k\ .
\end{align}
For type BC we get
\begin{align}\label{formrecMpmdtilde}
\tilde M^\pm_d&=\left(\prod_{j=1}^d1-st^j\right)+s\sum_{k=1}^dt^k\left(\prod_{j=0}^{k-1}\frac{1-q^{2d-2j}}{1-q^{k-j}}\right)\left(\prod_{j=k+1}^d1-st^j\right)
\tilde M_k\ .
\end{align}
For type D we get
\begin{align}\label{formrecMDdtilde}
\tilde M^D_d&=\left(\prod_{j=1}^d1-st^j\right)+\\
&\quad +s\sum_{k=1}^dt^k{d\choose k}_q\left(\sum_{l=0}^{\lfloor k/2\rfloor}
{k\choose 2l}_qq^{l(2d+2l-2k-1)}\right)\left(\prod_{j=k+1}^d1-st^j\right)
\tilde M_k\ .
\end{align}
\end{thm}

Proofs for Theorem \ref{thmEuler} are straightforward generalizations
of proofs for Corollaries \ref{correc}, \ref{corformulaMdpm} and
\ref{corformMD}.

Observe that $\tilde M_d^*$ incorporates the so-called Euler statistic counting
descents for type A and type BC. For type D, the polynomials $\tilde M_d^D$
incorporate slightly different statistic.

{\bf Acknowledgements.} I thank Michel Brion, Pierre de la Harpe, Emmanuel Peyre
and ??????????????????????????????????????
%M. Decauwert,  and the anonymous referee 
for corrections, interesting discussions, remarks or comments.

\noindent Roland BACHER, 
 Univ. Grenoble Alpes, Institut Fourier, 
 F-38000 Grenoble, France.
%\vskip0.5cm

\noindent e-mail: Roland.Bacher@univ-grenoble-alpes.fr

%\end{document}

\end{document}